\documentclass{jgcc} 
\usepackage{hyperref}
\usepackage{enumitem}
\usepackage{amssymb}
\usepackage{tikz-cd}
\usepackage{graphicx}

\pdfoutput=1

\usepackage{lastpage}

\jgccdoi{15}{2}{1}{12503}

\jgccheading{}{\pageref{LastPage}}{}{}{Nov.~3,~2023}{Mar.~31,~2024}{}  


\usepackage{hyperref}
\theoremstyle{plain} 

\begin{document}

\title[Cayley Linear--Time Computable Groups]{Cayley Linear--Time Computable Groups}

\author[P.~Kruengthomya]{Prohrak Kruengthomya}	
\address{Department of Mathematics, Faculty of Science, Mahidol University, Bangkok, 10400, Thailand}	
\email{prohrakju@gmail.com}  

\author[D.~Berdinsky]{Dmitry Berdinsky}	
\address{Department of Mathematics, Faculty of Science, Mahidol University and Centre of Excellence in Mathematics, Commission on Higher Education, Bangkok, 10400, Thailand}	
\email{berdinsky@gmail.com}  






\begin{abstract}
  \noindent This paper looks at 
  the class 
  of groups 
  admitting normal 
  forms for which the right 
  multiplication by 
  a group element is 
  computed in linear time
  on a multi--tape Turing 
  machine.  
  We show that the 
  groups 
  $\mathbb{Z}_2 \wr \mathbb{Z}^2$, 
  $\mathbb{Z}_2 \wr 
  \mathbb{F}_2$ and 
  Thompson's group 
  $F$ have normal forms 
  for which the right 
  multiplication by 
  a group element is 
  computed in linear 
  time on a 
  $2$--tape Turing machine. 
  This refines the 
  results 
  previously 
  established by Elder 
  and the authors
  that these groups 
  are Cayley polynomial--time 
  computable.
  
\end{abstract}

\maketitle


\section*{Introduction}

 Extensions of the notion of an automatic group 
introduced by Thurston and others \cite{Epsteinbook}
have been studied by different researchers. 
One of the extensions is the notion of  
a Cayley automatic group introduced by 
Kharlampovich, Khoussainov and Miasnikov \cite{KKM11}.
In their approach a normal form 
is defined by a bijection between a regular 
language and a group 
such that the right multiplication by a group element 
is recognized by a two--tape synchronous 
automaton.  
Elder and the authors looked at the further extension
of Cayley automatic groups allowing the language
of normal forms to be arbitrary (though it is always 
recursively enumerable \cite[Theorem~3]{BEK21}) but requiring 
the right multiplication by a group 
element to be computed by an automatic function 
(a function that can be computed by a two--tape 
synchronous automaton).
This extension is referred to as 
Cayley position--faithful (one--tape) 
linear--time
computable groups \cite[Definition~3]{BEK21}. 
These groups admit 
quasigeodesic normal forms 
(see Definition \ref{quasigeodesic_def})
for which the right multiplication 
by a group element is computed 
in linear time (on a position--faithful one--tape Turing
machine) and the normal form is computed 
in quadratic time 
\cite[Theorem~2]{BEK21}.  

In this paper we look at the groups 
admitting normal forms for which 
the right multiplication by a group element
is computed in linear time on a multi--tape
Turing machine (we refer to such groups as 
\emph{Cayley linear--time computable}). These normal forms are 
not necessarily quasigeodesic, see, e.g.,
the normal form of $\mathbb{Z}_2 \wr \mathbb{Z}^2$ 
considered in Section \ref{section_Z2wrZ2}. 
However, if such normal form is quasigeodesic, 
then it is computed in quadratic time  
(see Theorem \ref{quadratic_time_nf}), thus, 
fully retaining the basic algorithmic properties 
of normal forms for Cayley automatic groups: computability of the right multiplication by a group element in linear time and normal form in quadratic time.  
Cayley linear--time computable groups form a 
subset of Cayley polynomial--time 
computable groups introduced in 
\cite[Definition~5]{BEK21}, but clearly include 
all Cayley position--faithful (one--tape) 
linear--time
computable groups. 
We show that the groups 
$\mathbb{Z}_2 \wr \mathbb{Z}^2$, 
$\mathbb{Z}_2 \wr \mathbb{F}_2$ and Thompson's 
group $F$ are Cayley linear--time computable 
(on a $2$--tape Turing machine)
which refines the previous claims 
that these groups are Cayley polynomial--time 
computable \cite{BEK21}.   
To show that these three groups are 
Cayley linear--time computable we use the normal 
forms previously studied by the second author 
and Khoussainov for groups 
$\mathbb{Z}_2 \wr \mathbb{Z}^2$ and 
$\mathbb{Z}_2 \wr \mathbb{F}_2$ \cite{BK15} and Elder and Taback for Thompson's group $F$ \cite{ElderTabackThompson}. We note that 
\cite[Theorems~5,\,8]{BK15} and \cite[Theorem~3.6]{ElderTabackThompson} 
showing that $\mathbb{Z}_2 \wr \mathbb{F}_2$, 
$\mathbb{Z}_2 \wr \mathbb{Z}^2$ and Thompson's 
group $F$ are context--free, indexed and 
deterministic non--blind 1--counter graph automatic, respectively, do not imply that the right multiplications
by a group element for the normal forms considered in these groups are computed in linear time on a $2$--tape Turing machine. 
The latter requires careful verification that is  
done in this paper.    

Several researchers studied extensions 
of automatic groups utilizing 
different  
computational models.
Bridson and Gilman 
considered an extension of 
asynchronously automatic groups
using indexed languages \cite{BridsonGilman1996}. 
Baumslag, Shapiro and Short  extended the 
notion of an automatic group  
based on parallel computations by 
pushdown automata \cite{baumslagshapiro}.  
Brittenham and Hermiller  introduced autostackable groups which also extends the notion 
of an automatic group  
\cite{BrittenhamHermillerHolt14}. 
Elder and Taback introduced  
$\mathcal{C}$--graph automatic groups 
extending Cayley automatic groups 
and studied them for different classes
of languages $\mathcal{C}$ \cite{ElderTabackCgraph}.  
Jain, Khoussainov and Stephan introduced 
the class of a semiautomatic groups \cite{semiautomatic18} which generalizes the notion of a Cayley automatic group. 
Jain,  Moldagaliyev, Stephan and Tran studied extensions of Cayley automatic groups using transducers and tree automata \cite{lamplighter22}.

The paper is organized as follows. 
In Section \ref{section_preliminaries} we introduce 
the notion of a Cayley $k$--tape linear--time 
computable group. In Sections 
\ref{section_Z2wrZ2}, \ref{section_Z2wrF2} and 
\ref{section_F} we show that the wreath products 
$\mathbb{Z}_2 \wr \mathbb{Z}^2$, 
$\mathbb{Z}_2 \wr \mathbb{F}_2$ and 
Thompson's group $F$, respectively, are Cayley 
$2$--tape linear--time computable.   
Section \ref{conclusion_section} concludes the paper.  

\section{Preliminaries}
\label{section_preliminaries} 

In this section we introduce the notion of 
a Cayley linear--time computable group.  
We start with defining a basic concept
underlying this notion -- a \emph{function computed on a
    $k$--tape Turing machine in linear time}, where 
    $k>1$.    
\begin{defi}
	\label{pf_linear_time}
	
	A (position--faithful) 
	$k$--tape Turing machine is a 
	Turing machine with 
	$k$ semi--infinite tapes   
	for each of which the leftmost position contains the special symbol $\boxplus$ 
	which only occurs at this position and cannot be modified. 
	We denote by $\boxdot$ a special blank symbol, by $\Sigma$ the input alphabet for which $\Sigma\cap\{\boxplus ,\boxdot\}=\varnothing$ and by $\Gamma$ the tape alphabet for which $\Sigma \cup \{\boxplus ,\boxdot\} \subseteq  \Gamma$. Initially, for the input $x \in \Sigma^*$, the configuration of the first tape is 
	$\boxplus x \boxdot^\infty$ with the head being at the $\boxplus$ symbol. 
	The configurations of other $k-1$ tapes are 
	$\boxplus \boxdot^\infty$ with the 
	head pointing at the $\boxplus$ symbol. 
	During the computation the Turing machine operates as usual, reading and writing symbols from $\Gamma$ in cells to the right of the $\boxplus$ symbol.
	
	Let $k>1$. A function 
	$f: \Sigma^* \rightarrow \Sigma^*$
	is said to be computed on a 
	$k$--tape 
	Turing machine in linear time, if for the input 
	string $x \in \Sigma^*$ of length $n$  
	when started 
	with the first tape content being 
	$\boxplus x \boxdot^\infty$ and  
	other tapes content being 
	$\boxplus \boxdot^\infty$,  
	the heads pointing at 
	$\boxplus$, the Turing machine 
	reaches an accepting state and 
	halts in $Cn$ or fewer steps 
	with the first tape having 
	prefix $\boxplus f(x) \boxdot$, where
	$C>0$ is a constant. 
	There is no restriction on the 
	output beyond the first appearance of $\boxdot$ 
	on the first tape, the content of other 
	tapes and the positions of their heads.  
\end{defi}
In Definition \ref{pf_linear_time} position--faithfulness refers to a way the output in $\Sigma^*$
computed on a Turing machine is defined: 
it is the string 
$v \in \Sigma^*$ for which the content of the
first tape after a Turing machine halts 
is $\boxplus v \boxdot w \boxdot^\infty$, 
where $w$ is some string in $\Gamma^*$.
In general the output in $\Sigma^*$ computed on 
a Turing machine 
can be defined as the content 
of the first tape after it
halts with all symbols in 
$\Gamma \setminus \Sigma$ removed: see \cite{Papadimitrou} where the output of the computation 
on a one--tape Turing machine is defined as 
the string $y \in \Gamma^*$ for which 
the content of the tape after it halts is 
$\boxplus y \boxdot^\infty$, where $y$ is either empty
or the last symbol of $y$ is not $\boxdot$.  
For one--tape Turing machines the restrictions to position--faithful
ones matters -- there exist functions 
computed in linear time on a one--tape Turing machine 
which cannot be computed in 
linear time on a position--faithful one--tape Turing 
machine \cite{Stephan_lmcs_13}. The latter is 
due to the fact that shifting may require quadratic time. For multi--tape Turing machines ($k>1$) the restriction to position--faithful ones 
becomes irrelevant as shifting 
can always be done in linear time. 
Recall that a function 
$f : \Sigma^* \rightarrow \Sigma^*$ is 
called automatic if 
the language of convolutions 
$L_f = \{u \otimes v  \,|\, u, v \in \Sigma^*\}$ is 
regular. Case, Jain, Seah and Stephan showed 
that $f : \Sigma^* \rightarrow \Sigma^*$ 
is computed on a position--faithful one--tape 
Turing machine in linear time if and only if it is automatic \cite{Stephan_lmcs_13}.  
For $k>1$ the class of functions 
computed on $k$--tape Turing machines in linear 
time is clearly wider than the class of automatic functions.

Now let $G$ be a finitely generated group. 
Let $S= \{s_1, \dots, s_k\} \subset G$ 
be a set of its semigroup 
generators. That is, every group element of $G$ 
can be written as a product of elements in $S$. 
Below we define Cayley linear--time computable 
groups.   
\begin{defi} 
	\label{def_cyaley_lin}   
	Let $k>1$. We say that $G$ is Cayley  
	$k$--tape 
	linear--time computable 
	if there exist a language 
	$L \subseteq \Sigma^*$, a bijective mapping $\psi : L \rightarrow G$ and    
	functions 
	$f_s : \Sigma^* \rightarrow \Sigma^*$, 
	for $s \in S$,  each of which is computed 
	on a  
	$k$--tape Turing 
	machine in linear time, such that 
	for every $w \in L$ and $s \in S$:  $\psi(f_s (w)) = \psi (w) s$. That is, the following diagram 
	commutes:   
	\[ \begin{tikzcd}
		L \arrow{r}{f_s} \arrow[swap]{d}{\psi} & 
		L \arrow{d}{\psi} \\%
		G \arrow{r}{r_s}& G
	\end{tikzcd}, 
	\]  
	where $r_s : G \rightarrow G$ is the right 
	multiplication by $s$ in $G$: $r_s(g) = gs$ for
	all $g \in G$. 
\end{defi}
We refer to a bijective mapping 
$\psi : L \rightarrow G$ as a \emph{representation}. 
It defines a 
\emph{normal form} of a group $G$ which 
for every group element $g \in G$ 
assigns a unique string in $w \in L$ such that 
$\psi(w) = g$. For the latter we  also say 
that $w$ is a \emph{normal form of a group element} $g$. 
We will say that 
a representation $\psi : L \rightarrow G$ 
from Definition 
\ref{def_cyaley_lin}, as well as the corresponding 
normal form of $G$, are $k$--tape 
	linear--time computable. 
We note that Definition \ref{def_cyaley_lin} does not depend on the choice of a set of semigroup 
generators $S$ -- this follows directly from the observation that a composition of functions 
computed on 
$k$--tape Turing machines 
in linear time is also computed on a  
$k$--tape Turing machine in linear time.
We say that a group is Cayley linear--time computable 
if it is Cayley $k$--tape linear--time 
computable for some $k$. 

Cayley position--faithful (one--tape) linear--time 
computable groups were studied in  
\cite{BEK21}. 
They comprise wide classes of groups (e.g., all polycyclic groups), but at the same time 
retain all basic properties 
of Cayley automatic groups. Namely, each of such groups
admits a normal form for which 
the right multiplication by a fixed group element 
is computed in linear time and 
for a given word $g_1\dots g_n$, $g_i \in S$, the 
normal form of $g =g_1\dots g_n$ is computed in quadratic time. Furthermore, a position--faithful (one--tape) 
linear--time computable normal form is 
always \emph{quasigeodesic} \cite[Theorem~1]{BEK21} 
(see Definition \ref{quasigeodesic_def} for 
the notion of a quasigeodesic normal form 
introduced by Elder and Taback
\cite[Definition~4]{ElderTabackCgraph}).  
Moreover, this statement 
can be generalized to Theorem \ref{onetape_implies_quasigeodesic_thm}.
\begin{defi}
	\label{quasigeodesic_def}	
	A representation 
	$\psi: L \rightarrow G$ (a normal form of $G$)
	is said to be quasigeodesic  if 
	there exists a constant $C>0$
	such that for all $g \in G$:
	$|w| \leqslant C (d_S (g) + 1)$,   
	where $w$ is the normal form of $g$, 
	$|w|$ is its length and 
	$d_S(g)$ is the length of a shortest word 
	$g_1 \dots g_n$, $g_i \in S$, for which 
	$g=g_1 \dots g_n$ in $G$.    
\end{defi}
     
\begin{thm}
	\label{onetape_implies_quasigeodesic_thm}	
	A one--tape $o(n \log n)$--time computable
	normal form  is quasigeodesic.  
\end{thm}	
\begin{proof} 
	Let $\psi : L \rightarrow G$ be a bijection between a language $L \subseteq \Sigma^*$ 
	and a group $G$ defining a Cayley one--tape $o(n \log n)$--time computable normal form of $G$. 
	Let $S \subset G$ be a finite 
	set of semigroup generators. For each $s \in S$ 
	there exists a one--tape $o(n \log n)$--time 
	computable function 
	$f_s : \Sigma^* \rightarrow \Sigma^*$ such that 
	$\psi(f_s (w)) = \psi (w)s$ for all $w \in L$.  
	For a given $s \in S$ we denote by $\mathrm{TM}_s$ a one--tape Turing machine computing the function $f_s$ in 
	$o(n \log n)$ time. For a given 
	$x \in \Sigma^*$ we denote by $\mathrm{TM}_s (x)$ 
	the output of the computation on $\mathrm{TM}_s$ for the input $x$: it is the string 
	$y$ over the tape alphabet of $\mathrm{TM}_s$ 
	for which the content of the tape after $\mathrm{TM}_s$ 
	halts is $\boxplus y \boxdot^\infty$, where
	$y$ is either empty or the last symbol of $y$ 
	is not $\boxdot$ \cite{Papadimitrou}.
	We denote by $\mathrm{TM}_s'$ a Turing machine 
	which works exactly like $\mathrm{TM}_s$, but
	writes the marked blank symbol $\mathring{\boxdot}$
	instead of $\boxdot$. 
	Let $\Sigma_s = 
	\left(\Gamma_s \cup \{\mathring{\boxdot} \}\right) 
	\setminus \{ \boxdot\}$, where $\Gamma_s$ is the 
	tape alphabet of $\mathrm{TM}_s$. 
	Below we show that the language 
	of convolutions 
	$L_s = \{x \otimes z  \, 
	| \,x \in \Sigma^* \wedge   z= \mathrm{TM}_s'(x) \in \Sigma_s^*\}$ is regular. 
	
	We first recall the notion of convolutions of  
	two strings $x \in \Sigma^*$ and $z \in \Sigma_s^*$.
	Let $\diamond$ be a padding symbol which 
	does not belong to the alphabet $\Sigma$ and
	$\Sigma_s$.
	The convolution $x \otimes z$ is the string
	of length $\max \{|x|,|z|\}$ for which the 
	$k$th symbol is $\sigma_1 \choose \sigma_2$,
	where $\sigma_1$ is the $k$th symbol of $x$ if 
	$k \leqslant |x|$ and $\diamond$ otherwise and 
	$\sigma_2$ is the $k$th symbol of $z$ if
	$k \leqslant |z|$ and $\diamond$ otherwise.   
	We denote by $\Sigma^* \otimes \Sigma_s^*$ 
	the language of all convolutions
	$\Sigma^* \otimes \Sigma_s^* = 
	\{x \otimes z \, | \, x \in \Sigma^* \wedge  
	z \in \Sigma_s^* \}$. 
	Note that at most one of $\sigma_1$ and $\sigma_2$ can be
	equal to $\diamond$, but never both.   
	Therefore, $\Sigma^* \otimes \Sigma_s^*$ is 
	a language over the alphabet 
	$\Sigma_s'$ consisting of all symbols  
	$\sigma_1 \choose \sigma_2$ for which 
	$\sigma_1 \in \Sigma \cup \{ \diamond\}$,  
	$\sigma_2 \in \Sigma_s \cup \{ \diamond\}$ and 
	$\sigma_1$ and $\sigma_2$ cannot be 
	equal to $\diamond$ simultaneously.  
	Let us describe a Turing machine $\mathrm{TM}_s''$ 
	recognizing the language $L_s$. For the sake of 
	convenience we assume that $\mathrm{TM}_s''$ 
	has two semi--infinite tapes with the heads on each of
	the two tapes moving only synchronously. 
	\begin{algo} 
		\label{nlogn_alg}
		Initially, 
		the input string over the alphabet 
		$\Sigma_s'$ 
		is written 
		on the convolution of two semi--infinite 
		tapes -- the first and the second component 
		for each symbol in $\Sigma_s'$ appears on 
		the first and the second tape, respectively.   
		For each of the two tapes 
		the head is over the first cell.
		\begin{enumerate}
			\item{First $\mathrm{TM}_s''$ 
				scans the input from left to right 
				checking  
				if the input is of the form  
				$x \otimes z \in \Sigma^* \otimes \Sigma_s^*$ for some $x \in \Sigma^*$ 
				and $z \in \Sigma_s^*$. If it is not, 
				the input is rejected.  
				Simultaneously, if 
				on one of the tapes 
				$\mathrm{TM}_s''$ reads $\diamond$,  
				it writes $\boxdot$. After the heads read 
				$\boxdot$ on both tapes detecting the end of 
				the input,  they return back to the 
				initial position.}
			\item{$\mathrm{TM}_s''$ works exactly like 
				$\mathrm{TM}_s'$ on the first tape until it halts
				ignoring the content of the second tape.  
				Then the heads go back to the initial position.}	  
			\item{ After 
				that $\mathrm{TM}_s''$ scans the content of both 
				tapes checking if the heads read the same 
				symbol. If the heads do not read the same symbol, 
				$\mathrm{TM}_s''$ halts rejecting the input. 
				When the heads reads $\boxdot$ on both tapes 
				$\mathrm{TM}_s''$ halts accepting the input.}	  
			
		\end{enumerate}	
		
	\end{algo}	    	
	Let $n = |x|$ and $m =  \max\{|x|,|z|\}$. 
	In the first stage of Algorithm \ref{nlogn_alg} 
	$\mathrm{TM}_s''$ makes $O(m)$ moves. In the second 
	stage it makes $o(n \log n)$ moves. As the length 
	of $\mathrm{TM}_s'(x)$ is at most 
	$o (n \log n)$, in the third stage 
	$\mathrm{TM}_s''$ makes at most $o(n \log n)$
	moves. Since $n \leqslant m$ we obtain that $\mathrm{TM}_s''$
	makes at most $o(m \log m)$ moves before it either 
	accepts or rejects the input $x \otimes z$.  
	As the heads of $\mathrm{TM}_s''$ move only
	synchronously, $\mathrm{TM}_s''$ works 
	exactly like a one--tape Turing machine recognizing 
	the language $L_s$ in time $o (m \log m)$, where 
	$m$ is a length of the input. 
	Recall that Hartmanis \cite[Theorem~2]{Hartmanis68} 
	and, independently, Trachtenbrot \cite{Trachtenbrot64}  
	showed that a language recognized by a one--tape Turing 
	machine in $o(m \log m)$ time is regular. 
	Therefore, the language $L_s$ is regular.
	Now, by the pumping lemma, there exists a constant 
	$C_s>0$ such that 
	$|z| \leqslant |x|+C_s$ for all 
	$x \in \Sigma^*$, where $z =\mathrm{TM}_s' (x)$. 
	We have: $|f_s(x)| \leqslant |\mathrm{TM}_s (x)| 
	\leqslant |\mathrm{TM}_s' (x)|$ for all 
	$x \in \Sigma^*$.
	Therefore, $|f_s (x)| \leqslant |x| + C_s$ for 
	all $x \in \Sigma^*$.  
	Let $C>0$ be some positive constant which is greater 
	than or equal to $C_s$ for every $s \in S$ and  
	$|w_0|$, where $\psi (w_0) = e$. For a given 
	$g \in G$ let $g_1 \dots g_k$, $g_i \in S$, 
	be a shortest word for which $g = g_1 \dots g_k$ 
	in $G$ and $w$ be the string for which 
	$\psi(w)=g$. 
	Clearly, we have that 
	$|w| \leqslant C k + |w_0| \leqslant C (d_S (g) + 1)$ which proves the theorem.          
\end{proof}	
If $k > 1$, a $k$--tape linear--time computable 
normal form is not necessarily quasigeodesic: 
in Section \ref{section_Z2wrZ2} we show that 
the normal form of the wreath product 
$\mathbb{Z}_2 \wr \mathbb{Z}^2$ constructed 
in \cite[Section~5]{BK15} is $2$--tape linear--time computable; but this normal form is not quasigeodesic 
\cite[Remark~9]{BK15}. 
However, if a $k$--tape linear--time computable normal 
form is quasigeodesic, then it satisfies 
the same basic algorithmic property as 
a position--faithful (one--tape) linear--time 
computable 
normal form -- it is computed in quadratic time 
\cite[Theorem~2]{BEK21}.  
Indeed, let $\psi : L \rightarrow G$ be a bijection 
between a language $L \subseteq \Sigma^*$ and a 
group $G$ defining a quasigeodesic $k$--tape linear--time 
computable normal form of $G$.  
Let $S \subset G$ be a finite set of semigroup 
generators.  

\begin{thm} 
	\label{quadratic_time_nf}   
	There is a quadratic--time algorithm which 
	for a given word $g_1 \dots g_n \in S^*$, 
	$g_i \in S$, computes the normal form 
	of the group element $g= g_1 \dots g_n \in G$ --
	the string $w \in L$ for which $\psi(w) = g$.     
\end{thm}	 
\begin{proof} 
	Let $w_i \in L$ be the normal form of 
	the group element $g_1 \dots g_i$:
	$\psi (w_i) = g_1 \dots g_i$ for $i =1,\dots,n$.  
	Let $w_0$ be the normal form of the identity: 
	$\psi (w_0) = e$.  
	For each $i=0,\dots,n-1$, the string 
	$w_{i+1}$ is computed from $w_i$ on a 
	$k$--tape Turing machine in 
	$O(|w_i|)$ time.
	Since the normal form is quasigeodesic, 
	$|w_i| \leqslant C (i + 1)$ for all  
	$i =0, \dots, n-1$. So  
	$w_{i+1}$ is computed from $w_i$
	in $O(i)$ time for all $i=0,\dots,n-1$.  
	Now an algorithm computing $w$ from a given 
	input $g_1\dots g_n$ is as follows.   	
	Starting from $w_0$ it consecutively 
	computes $w_1, w_2, \dots, w_{n-1}$ and $w_n$.  
	The running time for this algorithm is at most 
	$O(n^2)$.   	
\end{proof}	
As an immediate corollary of Theorem 
\ref{quadratic_time_nf} we obtain that
the word problem for a group $G$ which admits 
a quasigeodesic $k$--tape linear--time 
computable normal form is decidable in quadratic time. 

\section{The Wreath Product  
	$\mathbb{Z}_2 \wr \mathbb{Z}^2$} 
\label{section_Z2wrZ2}

In this section we will show that the 
group $\mathbb{Z}_2 \wr \mathbb{Z}^2$ is Cayley $2$--tape linear--time computable. 
Every group element of 
$\mathbb{Z}_2 \wr \mathbb{Z}^2$ 
can be written as a pair $(f,z)$, where 
$z \in \mathbb{Z}^2$ and
$f : \mathbb{Z}^2 \rightarrow  \mathbb{Z}_2$ 
is a function for which
$f(\xi)$ is the 
non--identity element of $\mathbb{Z}_2$
for at most finitely many $\xi \in \mathbb{Z}^2$. 
We denote by $c$ the non--identity element of 
$\mathbb{Z}_2$ and by $a = (1,0)$ and 
$b = (0,1)$ the generators of 
$\mathbb{Z}^2 = \{(x,y)\, 
|\, x,y \in \mathbb{Z} \}$. 
The group $\mathbb{Z}_2$ is 
canonically embedded in 
$\mathbb{Z}_2 \wr \mathbb{Z}^2$ 
by mapping $c$ to  
$(f_c,e)$, where 
$f_c : \mathbb{Z}^2 
\rightarrow \mathbb{Z}_2$ 
is a function for which 
$f_c (z') = e$ for all 
$z' \neq e$ and 
$f_c (e) = c$. 
The group $\mathbb{Z}^2$ is 
canonically embedded in 
$\mathbb{Z}_2 \wr \mathbb{Z}^2$ 
by mapping 
$z \in \mathbb{Z}^2$ to 
$(f_e,z)$, where 
$f_e (z') = e$ for all  
$z' \in \mathbb{Z}^2$.
Therefore, we can identify 
$a,b$ and $c$ with the 
corresponding group element 
$(f_e,a)$, $(f_e,b)$ and 
$(f_c,e)$ in 
$\mathbb{Z}_2 \wr \mathbb{Z}^2$, 
respectively. The group 
$\mathbb{Z}_2 \wr \mathbb{Z}^2$ 
is generated by $a,b$ and $c$, 
so $S = \{a, a^{-1}, b, b^{-1}, c\}$
is a set of its semigroup generators. 
The formulas for the right 
multiplication 
in 
$\mathbb{Z}_2 \wr \mathbb{Z}^2$
by $a$, $a^{-1}$, $b$, $b^{-1}$ and $c$ 
are as follows. For a given 
$g = (f,z) \in \mathbb{Z}_2 \wr 
\mathbb{Z}^2$, where $z = (x,y)$,  
$ga = (f,z_1)$ for $z_1  = (x+1,y)$, 
$ga^{-1} = (f, z_1')$ for 
$z_1' = (x-1,y)$, 
$gb = (f,z_2)$ for 
$z_2 = (x,y+1)$, 
$g b^{-1} = (f, z_2')$ for 
$z_2 ' = (x,y-1)$ and 
$gc = (f',z)$, where 
$f'(z') = f(z')$ for all $z' \neq z$ 
and $f'(z) = f(z)c$.    

\begin{figure}
	\centering 	
	\includegraphics[width=4cm]{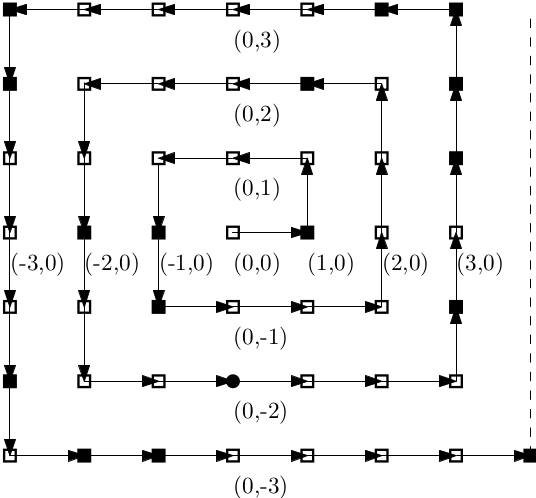}
	\caption{{\small An infinite digraph $\Gamma$ and 
			an element $h \in \mathbb{Z}_2 \wr \mathbb{Z}^2$.}}
	\label{z2wrz2fig1}
\end{figure}

 \emph{Normal form.}
We will use a normal form 
for elements of 
$\mathbb{Z}_2 \wr \mathbb{Z}^2$
described in \cite{BK15}. 
Let  $\Gamma$ be an infinite directed graph
shown on Fig.~\ref{z2wrz2fig1} which is
isomorphic to $\left(\mathbb{N};\mathrm{S}\right)$, 
where 
$\mathrm{S} : \mathbb{N} \rightarrow \mathbb{N}$ 
is the successor 
function $\mathrm{S}(n) = n + 1$.  
The vertices of $\Gamma$
are identified with  
elements of $\mathbb{Z}^2$;  
each vertex of $\Gamma$, except $(0,0)$, has 
exactly one ingoing and one outgoing edges and 
the vertex $(0,0)$ has one outgoing 
edge and no ingoing edges.  
Let $t : \mathbb{N} \rightarrow \mathbb{Z}^2$ 
be a mapping defined as follows: 
$t(1) = (0,0)$ and, for $k>1$, $t(k) = (x,y)$ 
is the end vertex of a directed path in 
$\Gamma$ of length $k-1$ which starts 
at the vertex $(0,0)$.  

We denote by $\Sigma$ the alphabet 
$\Sigma = \{0,1, C_0, C_1 \}$. 
Let  $g = (f,z)$ be an element
of the group $\mathbb{Z}_2 \wr \mathbb{Z}^2$. 
We denote by 
$r$ a number  
for which $t(r) = z$. Let
$m = \max \{ k \, | \, f (t (k)) = 1\}$ and 
$\ell = \max \{m,r\}$. 
A normal form $w \in \Sigma^*$
of the group  element $g$ is 
defined to be 
a string $w = \sigma_1 \dots \sigma_\ell$ 
of length $\ell$ for which 
$\sigma_k = 0$, if $f(t(k))=0$ and $k \neq r$,
$\sigma_k = 1$ if $f(t(k))=1$ and $k \neq r$,
$\sigma_k = C_0$ if $f(t(k))=0$ and $k = r$,
$\sigma_k = C_1$ if $f(t(k))=1$ and $k = r$.
For an illustration consider a group element 
$h \in \mathbb{Z}_2 \wr \mathbb{Z}^2$ shown 
on Fig.~\ref{z2wrz2fig1}: 
a white square indicates that the value of a function $f$ at a given point is $e$, 
a black square indicates that it is $c$, 
a black disk at the point $p= (0,-2)$ indicates 
that $f(p)=c$ and 
it specifies the position of the lamplighter.  
A normal form of the group element $h$ is 
the string:
$0100011000000100001000C_1000101111000011000101100001$. 

We denote by $L \subseteq \Sigma^*$  
a language of all such normal forms. 
The described normal form of 
$\mathbb{Z}_2 \wr \mathbb{Z}^2$ 
defines a bijection 
$\psi :   L \rightarrow \mathbb{Z}_2 \wr  \mathbb{Z}^2$  
mapping $w \in L$ to the 
corresponding group element 
$g  \in \mathbb{Z}_2 \wr \mathbb{Z}^2$. 
This normal form is not quasigeodesic 
\cite[Remark~9]{BK15}.

 \emph{Construction of Turing machines computing  the right multiplication in 
	$\mathbb{Z}_2 \wr \mathbb{Z}^2$
	by $a^{\pm 1}, b^{\pm 1}$ and $c$.}    
For the right multiplication  in 
$\mathbb{Z}_2 \wr \mathbb{Z}^2$ by $c$,  
consider a one--tape Turing machine  which 
reads the input $u \in \Sigma^*$ from left to 
right and when the head reads a symbol 
$C_0$ or $C_1$ it changes it  
to  $C_1$ or $C_0$, respectively, and then it halts.
If the head reads the blank symbol 
$\boxdot$, which indicates that the 
input $u$ has been read, it halts. 
The described Turing machine halts 
in linear time for every input 
$u \in \Sigma^*$. Moreover, if the input
$u \in L$, it computes the output 
$v \in L$ for which $\psi (u)c = \psi (v)$.

\begin{figure}
	\centering 	
	\includegraphics[width=5cm]{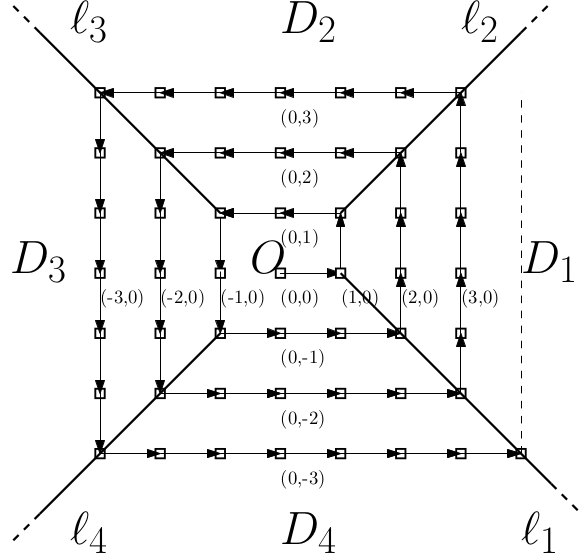}
	\caption{{\small A digraph $\Gamma$ and 
			subsets of $\mathbb{Z}^2$:   
			$O$, $\ell_1$, $\ell_2$, 
			$\ell_3$, $\ell_4$, $D_1$, 
			$D_2$, $D_3$ and $D_4$.}}
	\label{z2wrz2fig2}
\end{figure}

Let us describe a two--tape Turing 
machine computing the right 
multiplication by $a$ 
in  $\mathbb{Z}_2 \wr \mathbb{Z}^2$ 
which halts in linear time on every 
input in $u \in \Sigma ^*$.
We refer to this Turing 
machine as $\mathrm{TM}_a$.    
A key idea for constructing $\mathrm{TM}_a$ 
is to divide 
$\mathbb{Z}^2 = \{(x,y)\, | \, x,y \in \mathbb{Z}\}$ 
into nine subsets
$O$, $\ell_1$, $\ell_2$, $\ell_3$, $\ell_4$, $D_1$, $D_2$, $D_3$ and $D_4$ shown on Fig.~\ref{z2wrz2fig2}: 
\begin{itemize}	
	\item{$O = \{(0,0)\}$,}
	\item{$\ell_1 = \{\left(x,-(x-1)\right) \in 
		\mathbb{Z}^2 \,|\, 	x > 0 \}$,}
	\item{$\ell_2 = \{\left(x,x \right) \in \mathbb{Z}^2 \,|\,x > 0 \}$,}
	\item{$\ell_3 = \{\left(-x,x\right) \in \mathbb{Z}^2 \,|\, x > 0 \}$,}
	\item{$\ell_4 = \{\left(-x,-x\right) \in \mathbb{Z}^2 \,|\, 	x > 0 \}$,}
	\item{$D_1 = \{ \left(x, y \right) \in  
		\mathbb{Z}^2 \,|\, -(x-1) < y < x, x > 1\}$,}
	\item{$D_2 = \{\left(x,y\right) \in 
		\mathbb{Z}^2 \,|\, 	-y < x < y, y> 0 \}$,}
	\item{$D_3 = \{ \left(x,y\right) \in   
		\mathbb{Z}^2 \, | \, x < y < -x, x <0 \},$}
	\item{$D_4 = \{ \left(x,y\right) \in \mathbb{Z}^2\, | \, y <  x < -y+1,  y <0\}$.}
\end{itemize}	 
For a given $k \in \mathbb{N}$ we denote by 
$i$ the number of turns  
around the point $(0,0)$
a cursor makes
when moving along
the graph $\Gamma$ from the vertex  
$t(1)$ to the vertex $t(k)$.  
Formally, $i$ is defined as 
follows. 
Let $k_j = (j+1, -j)$ for $j \geqslant 1$.
If $k_j \leqslant k < k_{j+1}$, we 
put $i=j$; if $k < k_1$, we put $i=0$. 

Now we notice the following. If $t(k) \in \ell_1$, 
then $t(m) \in D_1$ for $k < m < k + (2i+1)$ 
and $t(m) \in \ell_2$ for 
$m = k + (2i + 1)$. If $t(k) \in \ell_2$, 
then $t(m) \in D_2$ for 
$k < m < k + (2i + 2)$ and 
$t(m) \in \ell_3$ for $m = k + (2i + 2)$. 
If $t(k) \in \ell_3$, then 
$t(m) \in D_3$ for $k < m < (2i +2)$ and
$t(m) \in \ell_4$ for $m = k + (2i +2)$.  
If $t(k) \in \ell_4$, then $t(m) \in D_4$ 
for $k < m < k + (2i  + 3)$ and 
$t(m) \in \ell_1$ for $m = k + (2i + 3)$. 
These observations ensure the correctness of the 
stage 3  of Algorithm \ref{Z2wrZ2_alg1}.

\begin{algo}[First iteration]   	  
	\label{Z2wrZ2_alg1}  
	Initially for $\mathrm{TM}_a$ a content of the 
	first tape is $\boxplus u \boxdot^\infty$ with 
	a head over $\boxplus$. A content of the second 
	tape is $\boxplus \boxdot^\infty$ with a head over 
	$\boxplus$.    
	In the first iteration 
	$\mathrm{TM}_a$ moves a head associated to the first tape from left to right until it 
	reads the symbols $C_0$ or $C_1$ 	  
	each time identifying 
	the set  
	$O$, $\ell_1$, $\ell_2$, $\ell_3$, $\ell_4$, 
	$D_1$, $D_2$, $D_3$ or $D_4$
	which contains $t(k)$ for the $k$th symbol 
	of $u$ being read.
	The second tape of
	$\mathrm{TM}_a$ 
	is used for counting  
	the number of turns $i$ when $t(k)$ moves along 
	a spiral formed by a graph 
	$\Gamma$.    	   
	Formally, in the first iteration 
	$\mathrm{TM}_a$ works as 
	follows until 
	the head associated to the first tape 
	reads either $C_0,C_1$ or $\boxdot$.  
	Let $S$ be a variable 
	which can take values only in the set $\{O, \ell_1, \ell_2, \ell_3, \ell_4, 
	D_1, D_2, D_3, D_4\}$.  
	\begin{enumerate}
		\item{$\mathrm{TM}_a$ reads the first $9$
			symbols of $u$ setting 
			$S$ to $O$, 
			$\ell_1$, $\ell_2$, $D_2$, $\ell_3$, 
			$D_3$, $\ell_4$, and $D_4$ when 
			the head reads the $k$th symbol 
			of $u$ for $k=1,2,3,4,5,6,7,8$, 
			respectively.}
		\item{$\mathrm{TM}_a$ reads the $10$th 
			symbol  of $u$ and set $S=\ell_1$. 
			On the second tape $\mathrm{TM}_a$ 
			moves the head right to the next cell 
			and writes the symbol $T$
			used for storing 
			the number of turns $i$.}		
		\item{The following steps are 
			repeated in loop one after another. 
			\begin{enumerate} 
				\item{$\mathrm{TM}_a$ reads the next 
					symbol of $u$,  
					moves the head associated to the 
					second tape left to the 
					previous cell 
					and set $S=D_1$. 
					Then $\mathrm{TM}_a$ keeps reading $u$ 
					and, simultaneously, 
					on the second tape it
					moves the head 
					first left until it reads 
					$\boxplus$ and then right until
					it reads $\boxdot$. 
					Then it sets 
					$S= \ell_2$.}
				\item{$\mathrm{TM}_a$ reads the next symbol 
					of $u$, moves the head 
					associated to the second tape left to 
					the previous cell and set $S=D_2$. Then 
					$\mathrm{TM}_a$ keeps reading $u$ and,
					simultaneously, on the second tape it
					moves the head first left until it reads $\boxplus$ and then right until it 
					reads $\boxdot$. Then it sets $S=\ell_3$.} 
				
				\item{$\mathrm{TM}_a$ reads the next symbol of 
					$u$, moves the head 
					associated to the second tape left to 
					the previous cell and set $S=D_3$. Then 
					$\mathrm{TM}_a$ keeps reading $u$ and,
					simultaneously, on the second tape  
					it moves the head first left until it 
					reads $\boxplus$ and then right until 
					it reads $\boxdot$. Then it sets 
					$S=\ell_4$.} 
				
				\item{$\mathrm{TM}_a$ reads the next symbol of $u$, moves the head associated to the second tape left to the previous cell and set $S=D_4$.
					Then $\mathrm{TM}_a$ keeps reading $u$ and, 
					simultaneously, on the second tape it moves 
					the head first left until it reads $\boxplus$ 
					and then right until it reads $\boxdot$.
					Then $\mathrm{TM}_a$ reads the next symbol 
					of $u$, writes $T$ on the second tape 
					and set $S=\ell_1$.}
		\end{enumerate}}      	    	     	      
	\end{enumerate}	
	If the head associated to the first tape reads 
	$\boxdot$, $\mathrm{TM}_a$ halts. 
	If it reads $C_0$ or $C_1$, $\mathrm{TM}_a$ 
	checks if the head associated to the second 
	tape reads $\boxdot$. If the symbol it reads is 
	not $\boxdot$, on the seconds tape $\mathrm{TM}_a$ moves the head right until it reads $\boxdot$.  
	Finally, unless $\mathrm{TM}_a$ halts, the 
	content of the first tape is 
	$\boxplus u \boxdot^\infty$ with the head over 
	$C_0$ or $C_1$ symbol and the content of the 
	second tape is $\boxplus T^i \boxdot^\infty$ 
	with the head over the first $\boxdot$ symbol. 
\end{algo}
The right multiplication 
of $g= (f,z) \in \mathbb{Z}_2 \wr \mathbb{Z}^2$
by $a$ changes a position of the lamplighter $z$ 
mapping $g$ to $ga = (f,z')$, where 
$z = (x,y)$ and
$z' = (x+1,y)$. 
Let $k$ and $k'$ be the integers 
for which $t(k)=z$ and $t(k')=z'$.  
Now we notice the following. 
If $z \in D_4 \cup \ell_4$, then 
$k ' = k + 1$.
If $z \in \ell_1 \cup D_1 \cup \ell_2 $, 
then $k' = k +  (8i + 9)$. 
If $z \in D_2 \cup \ell_3$, 
then $k ' = k - 1$.
If $z \in D_3$, then  
$k' = k - (8i + 5)$.  
So for the second iteration 
there are four cases to consider:
$S \in \{O,D_4,\ell_4 \}$, $S \in \{\ell_1, D_1, 
\ell_2 \}$, $S \in \{D_2, \ell_3\}$ and
$S = D_3$.

\noindent {\bf Case 1.} Suppose 
$S \in \{O, D_4, \ell_4\}$. 
On the first tape $\mathrm{TM}_a$ writes 
$0$ or $1$, if the head reads $C_0$ or $C_1$, 
respectively. Then the head moves right to the 
next cell. If the head reads $0$ or 
$\boxdot$, it writes $C_0$. 
If the head reads $1$, 
it writes $C_1$. Finally $\mathrm{TM}_a$ halts.

\noindent {\bf Case 2.} Suppose 
$S \in \{\ell_1, D_1, \ell_2\}$.
We divide a routine for this case into three stages. 
\begin{enumerate}[1.]
	\item{On the first tape $\mathrm{TM}_a$ 
		writes $0$ or $1$, if the head reads $C_0$ or $C_1$, respectively.}
	\item{The following subroutine is 
		repeated four times:
		\begin{enumerate}[(a)]
			\item{The head associated to the first tape 
				moves right to the next cell. 
				If the head reads $\boxdot$, it writes $0$. 
				The head associated to the second tape 
				moves left to the previous cell. } 
			\item{The head associated to the first tape keeps moving to the right writing $0$ if it reads 
				$\boxdot$. Simultaneously, the head associated to the second tape moves first left until it reads $\boxplus$ and then right until it reads 
				$\boxdot$.}       
	\end{enumerate}}
	\item{The head associated to the first tape moves 
		right to the next cell making the last 
		$(8i+9)$th move. If the head reads $0$ or 
		$\boxdot$, it writes $C_0$. 
		If the head reads $1$, 
		it writes $C_1$. Finally $\mathrm{TM}_a$ halts.} 
\end{enumerate}

\noindent {\bf Case 3.} Suppose 
$S \in \{D_2, \ell_3 \}$.  
We divide a routine 
for this case into two stages.
\begin{enumerate}[1.]
	\item{Let $\mathrm{ERASE}$ be a boolean variable.
		If the head associated to the first tape 
		reads $C_1$, $\mathrm{TM}_a$ sets $\mathrm{ERASE} = \mathrm{false}$. 
		If it reads $C_0$, the head moves right to the 
		next cell. If it reads $\boxdot$, 
		$\mathrm{TM}_a$ sets $\mathrm{ERASE} = \mathrm{true}$; otherwise, it sets $\mathrm{ERASE} = \mathrm{false}$.
		Then the head moves back to the previous cell. 
		Note that  $\mathrm{ERASE}=\mathrm{true}$ iff the head 
		reads the last symbol of $u$ which is $C_0$.   
		Now, if $\mathrm{ERASE}$, the head associated to 
		the first tape writes $\boxdot$.
		If $\mathrm{not} \,\, \mathrm{ERASE}$, the head 
		writes $0$ or $1$ if it reads $C_0$ or $C_1$, 
		respectively.} 
	\item{The head associated to the first tape moves 
		left to the previous cell.   	
		If the head  reads $0$ or $1$, it writes $C_0$ or $C_1$, respectively. Finally $\mathrm{TM}_a$ halts.} 
\end{enumerate}

\noindent {\bf Case 4.} Suppose $S=D_3$. We divide 
a routine for this into three stages.
\begin{enumerate}[1.]
	\item{The first stage is exactly the same as the 
		first stage in the case 3.}
	\item{The following subroutine is repeated 
		four times:
		\begin{enumerate}[(a)]
			\item{The head associated to the second tape moves 
				left to the previous cell.}
			\item{The head associated to the first tape 
				moves left. For each move, if 
				$\mathrm{ERASE}$ and the head 
				reads $0$, it writes $\boxdot$. 
				If $\mathrm{ERASE}$ and the head 
				reads $1$, $\mathrm{TM}_a$ sets 
				$\mathrm{ERASE}=\mathrm{false}$.
				Simultaneously, the head associated to 
				the second tape moves first left until 
				it reads $\boxplus$ and then right it 
				reads $\boxdot$.}      
	\end{enumerate}}     
	\item{The head associated to the first tape moves left to the previous cell making the last $(8i+5)$th move. If the head reads $0$ or $1$, it writes $C_0$ or $C_1$, respectively. Finally $\mathrm{TM}_a$ halts.}      	
\end{enumerate}
Clearly, the runtime of Algorithm \ref{Z2wrZ2_alg1} 
is linear. Moreover, for each of the cases 1--4 
the routine requires at most linear time. 
So $\mathrm{TM}_a$ halts in linear time for 
every input $u \in \Sigma^*$. If $u \in L$, 
$\mathrm{TM}_a$ halts with the output 
$v \in L$ written on the first tape 
for which $\psi(u)a = \psi(v)$. 
In the same way one can construct 
two--tape Turing machines computing the 
right multiplication by $a^{-1}$ and 
$b^{\pm 1}$ which halt in linear time on 
every input. Thus we have the following 
theorem. 

\begin{thm} 
	\label{Z2wrZ2_2_tape_thm}	
	The wreath product $\mathbb{Z}_2 \wr \mathbb{Z}^2$ 
	is Cayley $2$--tape linear--time computable. 
\end{thm}

 \section{The Wreath Product 
	$\mathbb{Z}_2 \wr \mathbb{F}_2$}
\label{section_Z2wrF2}

In this section we show that the group 
$\mathbb{Z}_2 \wr \mathbb{F}_2$ is 
Cayley $2$--tape linear--time computable.    
Every group element of 
$\mathbb{Z}_2 \wr \mathbb{F}_2$ 
can be written as a 
pair $(f,z)$, where $z \in \mathbb{F}_2$ and
$f : \mathbb{F}_2 \rightarrow \mathbb{Z}_2$ is a function for which $f(\xi)$ is the 
non--identity element of 
$\mathbb{Z}_2$ for at most finitely many 
$\xi \in \mathbb{F}_2$. 
We denote by $c$ the non--identity element of 
$\mathbb{Z}_2 = \{e,c\}$ and by $a,b$ 
the generators of 
$\mathbb{F}_2 = \langle a,b \rangle$. 
The group $\mathbb{Z}_2$ is canonically 
embedded in $\mathbb{Z}_2 \wr \mathbb{F}_2$ 
by mapping $c$ to $(f_c,e)$, where 
$f_c : \mathbb{F}_2 \rightarrow 
\mathbb{Z}_2$ is a function 
for which $f_c(x) = e$ for all 
$x \neq e$ and $f_c (e)=c$.  
The group $\mathbb{F}_2$ is canonically 
embedded in 
$\mathbb{Z}_2 \wr \mathbb{F}_2$ by 
mapping $x \in \mathbb{F}_2$ to 
$(f_e,x)$, where $f_e (y) =e$ for 
all $y \in \mathbb{F}_2$.          
Therefore, we can identify  
$a,b$ and $c$ with the corresponding 
group elements 
$(f_e,a)$, $(f_e,b)$ and 
$(f_c,e)$ in 
$\mathbb{Z}_2 \wr \mathbb{F}_2$, respectively. 
The group  $\mathbb{Z}_2 \wr \mathbb{F}_2$ 
is generated by $a,b$ and $c$, so 
$S = \{a,a^{-1},b,b^{-1},c\}$ is a set of 
its semigroup generators. 
The formulas for the right multiplication  
in $\mathbb{Z}_2 \wr \mathbb{F}_2$
by $a,a^{-1},b,b^{-1}$ and $c$ are as 
follows. For a given 
$g = (f,z) \in \mathbb{Z}_2 \wr \mathbb{F}_2$, 
$ga = (f,za)$, $ga^{-1} = (f,za^{-1})$, 
$gb = (f,zb)$, $gb^{-1} = (f,zb^{-1})$ 
and $gc = (f',z)$, where 
$f'(x)= f(x)$ for all $x \neq z$ and 
$f'(z) = f(z)c$.

\emph{Normal form.} We will use 
a normal form for elements of $\mathbb{Z}_2 \wr \mathbb{F}_2$ described in \cite{BK15}.    
Each element of $\mathbb{F}_2$ is 
identified uniquely with a reduced word over the alphabet $\{a,b,a^{-1},b^{-1}\}$.  
We denote by $F_a$ and $F_b$ the sets of elements of 
$\mathbb{F}_2$ for which the reduced words are of the form 
$a^{\pm 1}w$ and $b^{\pm 1}w$, respectively. 
Clearly, 
$\mathbb{F}_2  = F_a \cup F_b \cup \{ e \}$.  
Let $F_a'$ be a set of elements of $\mathbb{F}_2$ 
for which the reduced words are of the form 
$wa^{\pm1}$ or the empty word $\varepsilon$, and 
let $F_b'$ be a set of elements of $\mathbb{F}_2$ 
for which the reduced words are of the form 
$wb^{\pm1}$. Clearly, $\mathbb{F}_2 = F_a' \cup F_b'$.   
For given $s \in F_a'$ 
and $g = (f,z) \in \mathbb{Z}_2 \wr \mathbb{F}_2$,
we denote by $V_s$ the set 
$V_s = \{ p \in F_b \,| \, f(sp)=c \vee z=sp\}$. 
Similarly, for a given $t \in F_b'$ and 
$g = (f,z) \in \mathbb{Z}_2 \wr \mathbb{F}_2$,  
we denote by $H_t$ the set 
$H_t = \{p \in F_a \,|\, f(tp)=c \vee z=tp \}$.  
Let $\Sigma$ be the alphabet consisting of   
ten main symbols: $0$, $1$, $D_0$, $D_1$, $E_0$, 
$E_1$, $($, $)$, $[$, $]$ 
and fourteen  
additional symbols: 
$D_0 ^A$, $D_1 ^A$, $D_0 ^B$, $D_1 ^B$, $D_0 ^C$, $D_1 ^C$, $E_0^C$, $E_1^C$, $A_0$, $A_1$, $B_0$, $B_1$, $C_0$ and $C_1$.      
We will refer to the symbols 
$D_0,D_1,D_0 ^A, D_1 ^A, D_0 ^B, 
D_1 ^B, D_0 ^C, D_1 ^C$
and  
$E_0,E_1,E_0^C,E_1^C$ as 
$D$--symbols and $E$--symbols, 
respectively. 
A normal form of a given element 
$g \in \mathbb{Z}_2 \wr \mathbb{F}_2 $ is 
a string over the alphabet $\Sigma$
constructed in a  recursive way as follows. 
\begin{figure}
	\centering 	
	\includegraphics[width=4.5cm]{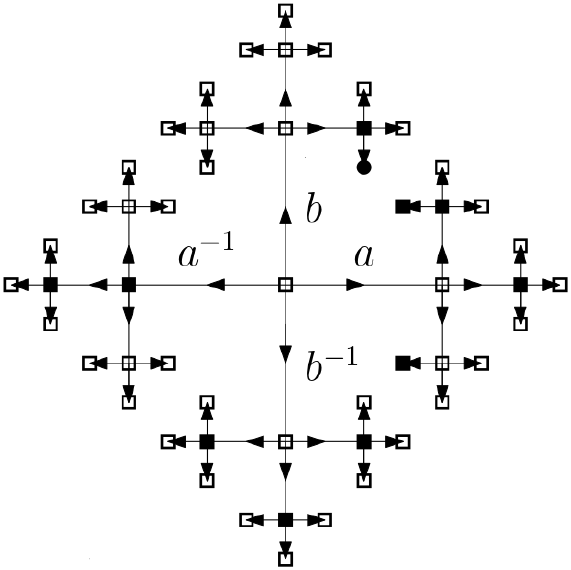}
	\caption{{\small A Cayley graph of   
			$\mathbb{F}_2$ and  
			an element 
			$(f,z) \in \mathbb{Z}_2 \wr 
			\mathbb{F}_2$. Black and 
			white squares indicate that a value of $f$ at a given point is $c$ and 
			$e$, respectively. 
			The 
			black disc 
			indicates a position of 
			the lamplighter $z$ and 
			that $f(z)=c$.}} 
	\label{freegraph2}
\end{figure}   

In the first iteration consider a cyclic subgroup  
$A = \{a^i \,|\, i \in \mathbb{Z}\} 
\leqslant \mathbb{F}_2$ which forms a horizontal 
line in a Cayley graph of $\mathbb{F}_2$ with respect 
to $a$ and $b$, see Fig.~\ref{freegraph2}.  
Scan this line from left to right checking for each $s \in A$  
whether or not $V_s \neq \varnothing$, $f(s)=c$, $s=z$ and $s=e$.  
In case $V_s \neq \varnothing$,  
$s \neq e$ and $s \neq z$, write the symbols $D_0$ or $D_1$, if 
$f(s)=e$ or $f(s)=c$, respectively. 
Similarly, in case $V_s \neq \varnothing$, $s=e$ and  $z \neq e$, 
write the symbols 
$D_0 ^A$ or $D_1 ^A$,
in case $V_s \neq \varnothing$, $s=e$ and  $z = e$, write the symbols $D_0 ^B$ or $D_1 ^B$ and in case $V_s \neq \varnothing$,  
$s = z$ and $s \neq e$, write the symbols 
$D_0 ^C$ or $D_1 ^C$, if 
$f(s)=e$ or $f(s)=c$, respectively.  
In case $V_s = \varnothing$, $s \neq e$ and $s \neq z$, write the symbols $0$ or $1$, 
if $f(s)=e$ or $f(s)=c$, respectively. 
Similarly, in case $V_s = \varnothing$, 
$s=e$ and  $z \neq e$, write the symbols 
$A_0$ or $A_1$, in case 
$V_s = \varnothing$, $s=e$ and  
$z = e$, write the symbols $B_0$ or $B_1$ and 
in case $V_s = \varnothing$, 
$s = z$ and $s \neq e$, write the symbol 
$C_0$ or $C_1$, if 
$f(s)=e$ or $f(s)=c$, respectively.  
Finally, for the obtained bi--infinite string
cut the infinite prefix and suffix 
consisting of $0$s, so the first and 
the last symbols of the remained finite string
are not $0$. 
For the element of 
$\mathbb{Z}_2 \wr \mathbb{F}_2$ shown
in Fig.~\ref{freegraph2} the resulted
string is $11D^A_0D_01$.

In the second iteration
each new $D$--symbol $\sigma$ is changed 
to a string of the form $(u^{-}\sigma u^{+})$, where the strings $u^{-}$ and $u^{+}$ are obtained as follows.
Every $D$--symbol corresponds to 
a group element $s \in \mathbb{F}_2$. 
In order to construct $u^{-}$ scan 
the elements on a vertical 
ray $B_s^{-} = \{sb^j\,|\,j<0\}$ 
from bottom to top checking
for each 
$t \in B_s^{-}$
whether or not 
$H_t \neq \varnothing$, $f(t)=c$ and 
$t=z$. In case $H_t \neq \varnothing$ and 
$t \neq z$, 
write the symbols $E_0$ or $E_1$, if 
$f(t)=e$ or $f(t)=c$, respectively. 
Similarly, in case $H_t \neq \varnothing$ and $t=z$, write the symbols $E_0 ^C$ or $E_1 ^C$,  
in case 
$H_t = \varnothing$ and $t \neq z$, write the symbols $0$ or $1$ and  
in case 
$H_t = \varnothing$ and  $t=z$, write the symbols 
$C_0$ or $C_1$, if $f(t)=e$ or $f(t)=c$, respectively. 
Finally, for the obtained infinite 
string cut the infinite prefix of $0$s, 
so the first symbol of the remained finite 
string $u^{-}$ is not $0$. 
In a similar way
a string $u^{+}$ is constructed by scanning the 
elements on a vertical ray 
$B_s^{+} = 
\{s b^j \, |\, j>0 \}$ from bottom to top and 
cutting the infinite suffix consisting of $0$s. 
For the element shown in Fig.~\ref{freegraph2} the resulted string is $11(1E_0D_0^AE_0)(E_0D_0E_1)1$.

In the third iteration each new 
$E$--symbol $\mu$ is changed to 
a string of the form 
$[v^{-}\mu v^{+}]$, where the strings
$v^{-}$ and $v^{+}$ are obtained as 
follows. Every $E$--symbol corresponds
to a group element 
$t \in \mathbb{F}_2$. 
In order to construct $v^{-}$ scan the 
elements on a horizontal ray 
$A_t^{-}= \{ta^{i}\,|\, i<0\}$
from left to right
checking 
for each  
$s \in A_t ^{-}$  whether or not 
$V_s \neq \varnothing$, 
$f(s)=c$ and $s=z$. 
In case $V_s \neq \varnothing$
and $s \neq z$, write the 
symbols $D_0$ or $D_1$, if $f(s)=e$ or 
$f(s)=c$, respectively. Similarly, in case $V_s \neq \varnothing$ and 
$s=z$, write the symbols 
$E_0^C$ or $E_1^C$, if $f(s)=e$ or 
$f(s)=c$, respectively. 
In case $V_s = \varnothing$ and $s \neq z$,
write the  symbols $0$ or $1$, if 
$f(s) = e$ or $f(s)=c$, respectively. 
Similarly, in case $V_s = \varnothing$ and 
$s=z$, write the symbols 
$C_0$ or $C_1$, if $f(s)=e$ or 
$f(s)=c$, respectively. Finally, 
for the obtained infinite string cut the infinite 
prefix consisting of $0$s, 
so the first symbol of the remained finite string $v^{-}$ is not 
$0$. In a similar way a string 
$v^{+}$ is constructed 
by scanning the 
elements on a horizontal ray 
$A_t ^{+}= \{ta^i\,|\,i>0 \}$ from 
left to right and cutting the infinite suffix 
consisting of $0$s. For the element in 
Fig.~\ref{freegraph2} the resulted 
string is 
$11(1[1E_01]D_0^A[E_0D_1])([1E_0]D_0
[1E_1])1$.  

This process is then repeated 
recursively until no new $D$ or 
$E$--symbols appear: 
for $i>1$ the $(2i)$th and the $(2i+1)$th 
iterations are performed exactly 
as the second and the third iterations described above, 
respectively. 
For the element 
in Fig.~\ref{freegraph2} the resulted string is 
$11(1[1E_01]D_0^A[E_0(C_1D_1)])
([1E_0]D_0 [1E_1])1$. 

We remark that the symbols 
$C_0, C_1, D_0 ^C, D_1 ^C, E_0 ^C, E_1 ^C$ and 
$A_0, A_1, D_0 ^A, 
D_1 ^A$
are used to mark the position of the 
lamplighter $z \in \mathbb{F}_2$ and 
the identity $e \in \mathbb{F}_2$, 
respectively,  
when $z \neq e$.  
The symbols $B_0,B_1,D_0 ^B, D_1 ^B$ 
are used to mark the position of the 
lamplighter and the identity when $z=e$. 

For a given group element 
$g = (f,z) \in 
\mathbb{Z}_2 \wr \mathbb{F}_2$, let 
$u \in \Sigma^*$ be a normal form of $g$ obtained 
by a recursive procedure described
above. We denote by $L \subseteq \Sigma^*$  
a language of all such normal forms. 
The described normal form of 
$\mathbb{Z}_2 \wr \mathbb{F}_2$ 
defines a bijection 
$\psi :   L \rightarrow \mathbb{Z}_2 \wr  \mathbb{F}_2$  
mapping $u \in L$ to the 
corresponding group element 
$g  \in \mathbb{Z}_2 \wr \mathbb{F}_2$. 
This normal form is quasigeodesic 
\cite[Theorem~5]{BK15}.

\emph{Construction of Turing machines computing  the right multiplication in 
	$\mathbb{Z}_2 \wr \mathbb{F}_2$
	by $a^{\pm 1}, b^{\pm 1}$ and $c$.}   
For the right multiplication 
in $\mathbb{Z}_2 \wr \mathbb{F}_2$ by $c$, 
consider a one--tape Turing machine 
which reads the input $u \in \Sigma^*$ from left 
to right and when the head reads a symbol
$D_i ^C$, $E_i ^C $, $D_i ^B$, $C_i$ 
or $B_i$, 
$i=0,1$, it changes this symbol to 
$D_j ^C$, $E_j ^C $, $D_j ^B$, $C_j$ 
or $B_j$, 
respectively, for $j= i +1 \mod\, 2$, 
and then it halts. If the head reads
the blank symbol $\boxdot$, which indicates 
that the input $u$ has been read, it halts.    
This Turing machine halts in linear time 
for every input $u \in \Sigma^*$. 
Moreover, if the input 
$u \in L$,  it
computes the output $v \in L$
for which $\psi (u) c =\psi (v)$.

Let us describe a two--tape Turing 
machine computing the right 
multiplication by $a$ 
in  $\mathbb{Z}_2 \wr \mathbb{F}_2$ 
which halts in linear time on every 
input in $u \in \Sigma ^*$.
We refer to this Turing 
machine as $\mathrm{TM}_a$.        
We refer to the symbols $C_0,C_1,D_0 ^C, D_1 ^C, E_0 ^C, E_1 ^C$ 
as $C$--symbols, $B_0,B_1,D_0 ^B, D_1 ^B$ as $B$--symbols and 
$(,),[,]$ as brackets symbols.

\begin{algo}[First iteration]  
	\label{Z2wrF2_alg_first_iteration}
	Initially for $\mathrm{TM}_a$ 
	a content of the first tape is 
	$\boxplus u \boxdot^\infty$ with a
	head over $\boxplus$. A content of the 
	second tape is $\boxplus \boxdot^\infty$
	with a head over $\boxplus$. 
	In the first iteration $\mathrm{TM}_a$ 
	moves a head associated to the first tape from 
	left to right until it reads either 
	$C$ or $B$--symbol. 
	The second tape is used 
	as a stack for storing the bracket symbols
	simultaneously checking their configuration:
	\begin{itemize}
		\item{If a head on the first tape reads 
			the symbol $($ or $[$, 
			on the second tape a head 
			moves right to the next cell
			and writes this symbol;} 
		\item{If a head on the first tape reads
			the symbol $)$ or $]$,
			$\mathrm{TM}_a$ checks if 
			a head on the second tape 
			reads $($ or $[$, respectively. If not, 
			$\mathrm{TM}_a$ halts.  
			Otherwise, on the second tape
			a head writes the blank symbol 
			$\boxdot$ and moves left to the previous cell.}
	\end{itemize}	
	If on the first tape a head   
	does not read a $C$ or $B$--symbol, 
	$\mathrm{TM}_a$ halts after a head reads a blank symbol 
	$\boxdot$  indicating that the input has been read.
\end{algo}

\noindent  Let $S$ be the symbol 
a head associated to the first tape 
reads at the end of the first iteration.
In the second iteration 
$\mathrm{TM}_a$ works depending
on the symbol $S$. There are three cases to consider: $S \in \{ D_0^C,D_1 ^C, D_0 ^B, 
D_1 ^B\}$, 
$S \in \{ E_0 ^C, E_1 ^C\}$ and 
$S \in \{C_0,C_1,B_0,B_1\}$.

\noindent  {\bf Case 1.} Suppose 
$S \in \{ D_0^C,D_1 ^C, D_0 ^B, D_1 ^B\}$. 
We  notice that if the input $u \in L$ and 
$S \in \{ D_0^C,D_1 ^C, D_0 ^B, D_1 ^B\}$, then on the second 
tape a head must read the left bracket symbol 
$($. So if it is not the left bracket symbol $($, 
$\mathrm{TM}_a$ halts. If it is 
the left bracket symbol $($, 
$\mathrm{TM}_a$ continues working as 
described in Algorithm \ref{alg_case_1}. 

\begin{algo}[Second iteration for Case 1]  
	\label{alg_case_1}   
	First $\mathrm{TM}_a$ marks the left bracket symbol $($ on the second tape 
	by changing it to 
	$\mathring{(}$.   
	On the first tape 
	it  changes the symbol $S$ 
	to $D_0$, $D_1$, $D_0 ^A$ and 
	$D_1^A$  if 
	$S = D_0 ^C$, $D_1 ^C$,  
	$D_0^B$ and 
	$D_1^B$, respectively. Then 
	$\mathrm{TM}_a$ proceeds as shown below.   	
	\begin{enumerate}
		\item{$\mathrm{TM}_a$
			keeps reading a content of the first tape 
			while on the second tape it works
			exactly as
			in Algorithm \ref{Z2wrF2_alg_first_iteration}
			until 
			on the first tape a head reads 
			the right bracket symbol $)$ and,
			at the same time, on  
			the second tape a head reads   
			the marked left bracket symbol  
			$\mathring{(}$. 
			If such situation does not occur, $\mathrm{TM}_a$ halts  
			after on the first tape a head reads a blank symbol 
			$\boxdot$.}
		
		\item{On the first tape a head moves
			right to the next cell and  
			on the second tape a head writes the blank 
			symbol $\boxdot$. Now let $S'$ be the symbol that a head associated to the first tape reads. If $S' \notin  
			\{ 0,1,E_0,E_1,A_0,A_1,\boxdot, ], ( \, \}$, $\mathrm{TM}_a$ halts. Otherwise, 
			depending on $S'$,
			$\mathrm{TM}_a$ proceeds  as follows.
			\begin{enumerate} 
				\item{ If $S' = 0, 1, E_0, E_1, A_0$ or
					$A_1$, $\mathrm{TM}_a$ changes it 
					to $C_0$, $C_1$, $E_0 ^C$, $E_1 ^C$, 
					$B_0$ or $B_1$, respectively, and 
					then halts.}        
				\item{If $S' = \boxdot$, $\mathrm{TM}_a$	changes it to $C_0$ and then halts.}
				\item{If $S' = ($, 
					on the second tape a head writes 
					the marked left bracket 
					$\mathring{(}$ and on the first
					tape a head moves right to the
					next cell. After that 
					$\mathrm{TM}_a$ 
					keeps reading a content of the first tape 
					while on the second tape it works
					exactly like 
					in Algorithm \ref{Z2wrF2_alg_first_iteration}
					until a head associated 
					to the first tape 
					reads a symbol 
					$Q= D_0, D_1, D_0^A$ or $D_1 ^A$ 
					and, at the same time, a head associated to the second tape reads 
					the marked left bracket 
					$\mathring{(}$. If such situation does not    occur, $\mathrm{TM}_a$ halts after  
					on the first tape a head reads a blank symbol $\boxdot$.
					Finally  $\mathrm{TM}_a$ changes 
					the symbol $Q$ to 
					$D_0 ^C$, $D_1 ^C$, 
					$D_0 ^B$ or $D_1 ^B$, if 
					$Q= D_0, D_1, D_0^A$ or $D_1 ^A$,
					respectively, and then it halts.}
				\item{If $S' = \, ]$, 
					$\mathrm{TM}_a$ 
					shifts  a non--blank content of the first 	tape starting  
					with this right bracket symbol $]$ by one position to the right
					and writes $C_0$ before it. 
					Then $\mathrm{TM}_a$ halts.}   
			\end{enumerate}
		}      	
		
	\end{enumerate}

\end{algo}	  

\noindent {\bf Case 2.} 
Suppose $S \in \{E_0 ^C, E_1 ^C \}$. 
$\mathrm{TM}_a$ continues working as described in Algorithm \ref{alg_case_2}. 
\begin{algo}[Second iteration for Case 2]   
	\label{alg_case_2}   
	On the first tape $\mathrm{TM}_a$ changes 
	$S$ to $E_0$ and $E_1$, if 
	$S = E_0 ^C$ and $S = E_1 ^C$,    
	respectively, and moves a head 
	right to the next cell.   
	Let $S'$ be the symbol that  
	a head associated to 
	the first tape reads. 
	If $S' \notin \{0,1,(,]\}$, 
	$\mathrm{TM}_a$ halts. Otherwise, 
	depending on $S'$, $\mathrm{TM}_a$ proceeds 
	as shown below. 
	\begin{enumerate}[(a)] 
		\item{If $S' =0$ or $1$,
			$\mathrm{TM}_a$ changes it to 
			$C_0$ or $C_1$, respectively, and 
			then halts.} 
		\item{If $S' = ($, 
			on the second tape a head writes 
			the marked left bracket 
			$\mathring{(}$ and on the first
			tape a head moves right to the next cell. Then  
			$\mathrm{TM}_a$ 
			keeps reading a content of the first tape 
			while on the second tape it works
			exactly as 
			in Algorithm \ref{Z2wrF2_alg_first_iteration}
			until a head associated 
			to the first tape 
			reads a symbol 
			$Q= D_0$ or $D_1$ 
			and, at the same time, a head associated to the second tape reads 
			the marked left bracket 
			$\mathring{(}$. If such situation does not  occur, $\mathrm{TM}_a$ halts after  
			on the first tape a head reads a blank symbol $\boxdot$.
			Finally  $\mathrm{TM}_a$ changes 
			the symbol $Q$ to 
			$D_0 ^C$ or $D_1 ^C$, if 
			$Q= D_0$ or $D_1$,
			respectively, and halts.}
		
		\item{If $S'= \,]$, $\mathrm{TM}_a$ 
			works exactly as in the 
			case $d$ for the second stage 
			of Algorithm \ref{alg_case_1}. 
		}
	\end{enumerate}	
\end{algo}

\noindent {\bf Case 3.}
Suppose $S \in  \{C_0, C_1, B_0, B_1 \}$.
Let $P$ be the symbol a head associated to 
the second tape reads. 
We notice that if the input $u \in L$ and 
$S \in \{C_0, C_1\}$, then
$P \in \{ (,[ , \boxplus \}$. 
If $u \in L$ and $S \in \{B_0, B_1\}$, 
then $P = \boxplus$. 
So if $S \in \{C_0, C_1\}$ and 
$P \notin \{ (,[ , \boxplus \}$ or 
$S \in \{B_0 ,B_1\}$  and
$P \neq \boxplus$, 
$\mathrm{TM}_a$ halts.
Otherwise, it continues working  
as described in Algorithm \ref{alg_case_3}.  

\begin{algo}[Second iteration for Case 3]
	\label{alg_case_3}
	
	Depending on $P$ and $S$, $\mathrm{TM}_a$	
	works as shown below. 
	\begin{enumerate}[(a)] 
		\item{If $P = ($, then $\mathrm{TM}_a$ changes 
			the symbol $S$ to the substring $w$ that is 
			equal to $[E_0C_0]$ and
			$[E_1 C_0]$, if $S =C_0$ and $S= C_1$, 
			respectively. This can be done by shifting the 
			non--blank content of the first tape following
			$S$ by three positions to the right
			and then wring $w$ before it; in particular, 
			the symbol $S$ will be overwritten by 
			the left bracket symbol $[$.}	
		\item{If $P = [$ and $S=C_1$, 
			$\mathrm{TM}_a$ changes $C_1$ to $1$ 
			and moves a head associated to the first 
			tape by one position to the right. 
			Let $S'$ be the symbol a head associated 
			to the first tape reads. 
			If $S' \notin \{0,1,E_0,E_1,(,]\}$, 
			then $\mathrm{TM}_a$ halts. Otherwise, 
			it continues working as shown below.
			\begin{itemize} 
				\item{If $S' \in \{0,1,E_0,E_1\}$, then 
					$\mathrm{TM}_a$ changes $S'$ to 
					$C_0,C_1, E_0 ^C, E_1 ^C $ if 
					$S' = 0,1, E_0, E_1$, respectively, and
					then it halts.}
				\item{If $S' = ($, then 
					$\mathrm{TM}_a$ works exactly as in
					the case b for Algorithm \ref{alg_case_2}.}
				\item{If $S' = \, ]$, $\mathrm{TM}_a$ 
					works exactly as in the 
					case $d$ for the second stage 
					of Algorithm \ref{alg_case_1}.} 	
			\end{itemize}	
		}
		\item{If $P = [$ and $S = C_0$, 
			$\mathrm{TM}_a$ moves a head 
			associated to the first tape 
			by one position to the left. 
			Let $T$ be the symbol the head 
			reads. Depending on  
			$T$, $\mathrm{TM}_a$ proceeds 
			as shown below. 
			\begin{itemize}
				\item{If $T \neq [$, $\mathrm{TM}_a$ 
					moves the head 
					by one position to the right, changes  $C_0$ to $0$ and moves 
					the head again 
					by one position to the right. 
					After that, depending on the symbol 
					$S'$ the head reads, it continues working exactly as in the case b  
					of the present algorithm.}  
				\item{If $T =[$, $\mathrm{TM}_a$  reads the two symbols following $C_0$.  Let $S'$ and  
					$S''$ be the first and the 
					seconds symbols following $C_0$, 
					respectively. 
					If $S' \in \{E_0, E_1\}$ and 
					$S''=\,]$, then 
					$\mathrm{TM}_a$ changes the 
					substring $[C_0 S']$ to the 
					symbol $C_0$ or $C_1$ if 
					$S'= E_0$ or $S' = E_1$, 
					respectively. This can be done 
					by shifting the non--blank content 
					of the first tape following the 
					substring $[C_0 S']$ by three 
					positions to the left and then 
					changing the left bracket symbol 
					$[$ to $C_0$ or $C_1$ if 
					$S' = E_0$ or $S' = E_1$, 
					respectively. 
					If $S' \notin \{E_0, E_1\}$ 
					or $S'' \neq \,]$, 
					$\mathrm{TM}_a$ shifts the 
					non--blank content of the 
					first tape following $C_0$
					by one position to the left. 
					Then, if $S' \in 
					\{0,1,E_0,E_1,(\}$,  
					$\mathrm{TM}_a$ continues working exactly as in the case b  
					of the present algorithm and halts 
					otherwise.} 
		\end{itemize} }
		\item{If $P = \boxplus$, 
			$\mathrm{TM}_a$ moves 
			a head associated to the 
			first tape by one position 
			to the left. 
			Let $T$ be the symbol 
			the head reads. 
			Depending on $S$ and $T$, 
			$\mathrm{TM}_a$ proceeds 
			as follows.  	      
			If $T \neq \boxplus$ or 
			$S \in \{ C_1, B_0 , B_1\}$,
			$\mathrm{TM}_a$ moves 
			the head by one position 
			to the right, 
			changes $S$ to $0,1, A_0$ 
			or $A_1$, if $S= C_0, C_1,
			B_0$ or $B_1$, respectively,  
			and moves the head 
			again by one position  
			to the right. 
			If $T = \boxplus$
			and $S = C_0$,  
			$\mathrm{TM}_a$
			shifts the non--blank 
			content of the first 
			tape following $C_0$ by 
			one position to the left
			erasing $C_0$ and places  
			the head over 
			the first symbol after 
			$\boxplus$.
			Now let $S'$ be a symbol 
			the head reads.  
			If   $S' \notin 
			\{0,1,A_0,A_1,\boxdot, (\}$, 
			$\mathrm{TM}_a$ halts.
			Otherwise, if  $S' \in \{0,1,A_0,A_1, 
			\boxdot\}$, 
			$\mathrm{TM}_a$ changes 
			$S'$ to $C_0,C_1,B_0,B_1$ and $C_0$ if 
			$S'= 0,1,A_0,A_1$ and 
			$\boxdot$, respectively, and then halts.    
			If $S' = ($, 
			$\mathrm{TM}_a$ moves 
			a head associated to 
			the second tape by one
			position to the right. 
			After that it continues 
			working exactly as in the case c for the stage 2 of  Algorithm \ref{alg_case_1}.}       
	\end{enumerate}		
\end{algo}	
The runtime of Algorithm  \ref{Z2wrF2_alg_first_iteration}
is linear. Also, as shifting a portion 
of a tape by a fixed number of positions 
requires at most linear time,  
for each of the  algorithms 
\ref{alg_case_1}--\ref{alg_case_3}  
the runtime is linear. 
Therefore, $\mathrm{TM}_a$ halts in 
linear time for every input $u \in \Sigma^*$.  
Moreover, if the input $u \in L$, 
$\mathrm{TM}_a$ halts with the output 
$v \in L$  for which $\psi (u)a  = \psi (v)$
written on the first tape. 
Two--tape Turing machines computing 
the right multiplication by $a^{-1}$ and 
$b^{\pm 1}$ which halt in linear time
one every input  are 
constructed in the same way as 
$\mathrm{TM}_a$  
with minor modifications. Thus we have the 
following theorem.  
\begin{thm} 
	\label{Z2wrF2_2_tape_thm}	 
	The wreath product 
	$\mathbb{Z}_2 \wr \mathbb{F}_2$ is Cayley 
	$2$--tape linear--time computable. 
\end{thm}	

\section{Thompson's Group $F$}  
\label{section_F}

In this section we show that Richard Thompson's 
group $F$ is Cayley $2$--tape linear--time computable. 
The  group 
$F = \langle x_0, x_1 \, | \, 
[x_0 ^{-1} x_1,x_0 ^{-1} x_1 x_0], 
[x_0^{-1} x_1, x_0^{-2} x_1 x_0 ^2]
\rangle$ 
admits the 
infinite presentation of the form:
\[ 
	F =  \langle x_0, x_1, x_2 \dots 
	\, | \,\, x_j x_i = x_i x_{j+1} \,\, \mathrm{for} \,\,
	i < j \rangle. 	
\]	
This infinite presentation provides a
standard infinite normal form for elements of 
$F$ with respect to  
generators  $x_i$, $i \geqslant 0$, 
discussed by Brown and Geoghegan 
\cite{Brown_Geoghegan}.   
Namely, applying the relations 
$x_j x_i = x_i x_{j+1}$ for $i <j$, 
a group element 
$g \in F$ can be 
written uniquely as: 
\[ 
	\label{inf_normal_form}	
	x_{i_0}^{e_0} x_{i_1}^{e_1} 
	\dots x_{i_m}^{e_m} 
	x_{j_n}^{-f_n} \dots 
	x_{j_1} ^{-f_1} x_{j_0} ^{-f_0},
\]	
where:
\begin{itemize} 
	\item{$0 \leqslant i_0
		< i_1 < i_2 <
		\dots < i_m$
		and 
		$0 \leqslant j_0 
		< j_1 < j_2 < \dots < j_n$;} 	
	\item{$e_i,f_j>0$ for 
		all $i,j$;}
	\item{if $x_i$ and
		$x_i^{-1}$ are 
		both present in 
		the expression, 
		then so is $x_{i+1}$ or $x_{i+1}^{-1}$.}	  
\end{itemize}	

For other equivalent interpretations of
Thompson's group $F$,  
as the set of piecewise
linear homomorphisms of 
the interval $[0,1]$ and
as the set of pairs of reduced finite rooted 
binary trees, we refer the reader to 
\cite{Thompsons_intro}.

\emph{Normal form.} Based on the standard 
infinite normal form \eqref{inf_normal_form}
Elder and Taback
\cite{ElderTabackThompson} 
constructed 
a normal form for elements of $F$
over the alphabet 
$\Sigma = \{a,b,\#\}$ as follows. 
Let $M=\max \{i_m,j_n\}$.
First let us rewrite 
\eqref{inf_normal_form} 
in the form 
such that 
every generator 
$x_i, i=0,\dots,M$ 
appears twice: 
\[ 
	\label{ET_norm_form_pre}
	x_0^{r_0} x_1^{r_1} 
	x_2^{r_2} \dots 
	x_M^{r_M} x_M^{-s_M}
	\dots x_2 ^{-s_2}
	x_1 ^{-s_1} x_0^{-s_0},
\]	
where $r_i,s_i \geqslant 0$, exactly one of 
$r_M,s_M$ is nonzero, 
and $r_i s_i >0$ 
implies 
$r_{i+1} + s_{i+1}>0$.  
After that we rewrite \eqref{ET_norm_form_pre} 
in the following form: 
\[
	\label{ET_norm_form_final}	
	a^{r_0} b^{s_0} \# a^{r_1} b^{s_1} \# \dots \# a^{r_M} b^{s_M},
\]	
where again $r_i,s_i \geqslant 0$, exactly one of 
$r_M,s_M$ is nonzero, 
and $r_i s_i >0$ implies 
$r_{i+1} + s_{i+1}>0$. 
We denote by $L_\infty$ the language of all strings 
of the form \eqref{ET_norm_form_final}.  
The language $L_\infty$ is regular 
\cite[Lemma~3.1]{ElderTabackThompson}.
Following the notation in \cite{ElderTabackThompson},
for a given $u = a^{r_0} b^{s_0} \# a^{r_1} b^{s_1} \# \dots \# a^{r_M} b^{s_M}$ 
from the language $L_\infty$ we denote by 
$\overline{u}$ the corresponding group element 
$x_0^{r_0} x_1^{r_1} 
x_2^{r_2} \dots 
x_M^{r_M} x_M^{-s_M}
\dots$ $x_2 ^{-s_2}
x_1 ^{-s_1} x_0^{-s_0}$ in $F$.  
The described normal form of 
$F$ defines a bijection between  
$L_\infty$ and $F$ mapping $u$ 
to $\overline{u}$.
This normal form is quasigeodesic \cite[Proposition~3.3]{ElderTabackThompson}

\emph{Construction of Turing machines computing 
	the right multiplication in $F$ by $x_0^{\pm 1}$ and 
	$x_1^{\pm 1}$.}     
By \cite[Proposition~3.4]{ElderTabackThompson}
the language 
$L_{x_0^{-1}} = \{\otimes (u , v) \, | \, 
u, v \in L_\infty, 
\overline{u} x_0 ^{-1} =_F 
\overline{v}\}$ is regular.                     
This implies that there is a linear--time 
algorithm that from a given input
$u \in L_\infty$ computes the output 
$v \in L_\infty$ such that   
$\overline{u} x_0 ^{-1} =_F 
\overline{v}$; see, e.g.,  
\cite[Theorem~2.3.10]{Epsteinbook}. 
Moreover, this algorithm can be 
done in linear time on 
a one--tape Turing machine 
\cite[Theorem~2.4]{Stephan_lmcs_13}. 
Thus we only have to analyze the right 
multiplication by $x_1^{\pm 1}$. 
Below we will show that the right 
multiplication in the group $F$ by $x_1 ^{\pm 1}$ 
can be computed in linear time on a $2$--tape 
Turing machine.


We denote by $w$ the infinite normal 
form \eqref{inf_normal_form} for a 
group element $g \in F$.
We denote by 
$u$ and $v$ the normal forms 
\eqref{ET_norm_form_final} for 
$g$ and $g x_1^{-1}$, respectively; 
that is, $\overline{u} =g$ and 
$\overline{v} = g x_1 ^{-1}$.  
Let us describe multi--tape\footnote{
	Describing
	$\mathrm{TM}_{x_1^{-1}}$ and $\mathrm{TM}_{x_1}$ 
	we allow them
	to have as many tapes as needed. However, later we notice 
	that two tapes are enough for computing the right 
	multiplication by $x_1^{-1}$ and $x_1$ in linear 
	time.} 
Turing machines
computing the right multiplication
by $x_1 ^{-1}$ and 
$x_1$ in $F$  which halts in linear time on 
every input in $\Sigma^*$. We refer to these Turing machines  
as $\mathrm{TM}_{x_1^{-1}}$ and $\mathrm{TM}_{x_1}$, respectively.  
Initially for $\mathrm{TM}_{x_1^{-1}}$  
a content 
of the first tape is 
$\boxplus u \boxdot^\infty$ with the
head over $\boxplus$. 
For $\mathrm{TM}_{x_1}$  
a content 
of the first tape is 
$\boxplus v \boxdot^\infty$ with the
head over $\boxplus$.
A content for each of the other tapes
is $\boxplus \boxdot^\infty$
with the head over $\boxplus$. 
We may assume that the input is in the regular language 
of normal forms $L_\infty$. This can be verified in 
linear time by 
reading the input on the first tape. 
If the input is not in $L_\infty$, a Turing machine halts. Otherwise, a head associated to the first tape returns to its initial 
position over the $\boxplus$ symbol.

The general descriptions of 
$\mathrm{TM}_{x_1^{-1}}$ and 
$\mathrm{TM}_{x_1}$ are as follows.  
For the input $u$ 
a Turing machine
$\mathrm{TM}_{x_1^{-1}}$
verifies each of the cases 
described by Elder and Taback in \cite[Proposition~3.5]{ElderTabackThompson}
one by one. 
Once it finds a valid case, it runs a 
subroutine computing $v$ from $u$ and 
writes  it on the first tape. Then $\mathrm{TM}_{x_1^{-1}}$
halts. 
For the input $v$ a Turing machine 
$\mathrm{TM}_{x_1}$ first copies it on
the second tape where it is stored until 
it halts.   
Then $\mathrm{TM}_{x_1}$ tries each of the 
cases one by one. For the case being tried 
it runs a subroutine computing $u$ 
from $v$ written on the second tape  
and writes the output $u$ 
on the first tape. Then it verifies whether 
or not
the case being tried is valid for $u$. If it is valid, 
then it runs the corresponding subroutine
for $\mathrm{TM}_{x_1^{-1}}$  
computing $v$ from $u$ 
and writes it on a third tape; otherwise, it tries the next case. Then $\mathrm{TM}_{x_1^{-1}}$ verifies whether or not the contents of 
the second and the third tapes are the same. If they are the same, then $\mathrm{TM}_{x_1^{-1}}$ halts; otherwise, $\mathrm{TM}_{x_1^{-1}}$ tries the next case. As there are only finitely 
many cases to try, $\mathrm{TM}_{x_1^{-1}}$ will halt with
the string $u$ written on the first tape.      
For each of the cases in \cite[Proposition~3.5]{ElderTabackThompson} we describe
a subroutine for its validity verification, 
a subroutine for computing $v$ from $u$ 
and a subroutine for computing $u$ from $v$.

\noindent {\bf Case 1}: Suppose that $s_0 = 0$ or, equivalently, 
the infinite normal form $w$ does not contain $x_0$ 
to a negative exponent.
That is, $u$ is either of the form $u = a^{r_0}\#\gamma$
for $\gamma \in \Sigma^*$ or $u = a^{r_0}$, where 
$r_0 \geqslant 0$.   
This case can be verified 
by reading $u$.   
There are the following three cases to consider.

\noindent {\bf Case 1.1}: 
The normal form $u$ is of the form $u = a^{r_0}$ for 
$r_0 \geqslant 0$.
This case can be verified by reading $u$.  
If $u = a^{r_0}$ for $r_0 \geqslant 0$, then 
$v = a^{r_0}\#b$.               
A subroutine for computing $v$ from $u$ 
appends the suffix $\#b$ to $u$.  
A subroutine for computing $u$ from $v$ 
erases the last two symbols of $v$ by writing the blank 
symbols $\boxdot\boxdot$.

\noindent {\bf Case 1.2}:  The normal from $u$ contains at least one 
$\#$ symbol and $wx_1^{-1}$ is the infinite 
normal form for $gx_1^{-1}$.           
The latter is true if at least
one of the following conditions holds.

\begin{enumerate}[(a)]	
	\item{The expression $w$ contains no $x_1$ 
		terms to a positive exponent: $r_1 = 0$;}
	
	\item{The expression $w$ contains $x_1$ to 
		a negative power: $s_1 \neq 0$;} 
	
	\item{The expression $w$ contains $x_2$ to 
		a nonzero power: $r_2 \neq 0$ or 
		$s_2 \neq 0$.}        
\end{enumerate}
Each of these three conditions can be verified  by reading $u$.  
If $u = a^{r_0} \# a^{r_1} b^{s_1} \gamma$,
where $\gamma$ is empty or begins with $\#$, 
then $v= a^{r_0} \# a^{r_1} b^{s_1+1} \gamma$. 
A subroutine for computing $v$ from $u$ 
shifts a suffix $b^{s_1} \gamma$ by one position 
to the right and writes the $b$ symbol before it. 
A subroutine for computing $u$ from $v$ 
shifts  the suffix $b^{s_1}\gamma$ by one position to the left.

\noindent {\bf Case 1.3}: The normal form $u$ is of the form $u  = a^{r_0} \# a^{r_1} \gamma$ with 
$r_1 > 0$ and $\gamma$ is either 
empty or $\gamma = \# \# \gamma'$
for $\gamma' \in \Sigma^*$. 
This case can be trivially verified  
by reading $u$. 
The infinite 
normal form $w$ is 
$w = x_0 ^{r_0} x_1 ^{r_1} \eta$, where 
$\eta$ is either empty 
or $\eta = x_i ^{r_i} \dots x_j ^{-s_j}$
for some $i,j > 2$.
Then 
$wx_1 ^{-1} = x_0 ^{r_0} x_1 ^{r_1} 
\eta x_1^{-1} = 
x_0 ^{r_0} x_1 ^{r_1} \eta'$, where 
$\eta'$ is obtained from 
$\eta$ replacing $x_i^{\pm 1}$ by 
$x_{i-1}^{\pm 1}$. Therefore, 
for $u  = a^{r_0} \# a^{r_1} \gamma$ 
we have the following three subcases.
\begin{enumerate}[(a)]
	\item{If $r_1>1$ 
		and $\gamma$ is empty, 
		then $v=a^{r_0} \# a^{r_1-1}$. 
		A subroutine for computing 
		$v$ from $u$ erases the last symbol 
		of $u$ by writing the blank 
		symbol $\boxdot$. A subroutine for 
		computing $u$ from $v$ appends the 
		$a$ symbol to $v$.} 
	\item{If $r_1=1$ and 
		$\gamma$ is empty, then 
		$v=a^{r_0}$. A subroutine for computing
		$v$ from $u$ erases the last two symbols
		of $u$ by writing $\boxdot\boxdot$. 
		A subroutine for computing $u$ from 
		$v$ appends $\#a$ to $v$. }
	\item{If $r_1 > 1$ and 
		$\gamma= \# \# \gamma'$, 
		then $v=a^{r_0}\#a^{r_1-1}\#\gamma'$. 
		A subroutine for computing $v$ from 
		$u$ shifts the suffix $\#\gamma'$
		by two positions to the left.
		A subroutine for computing $u$ from 
		$v$ shifts the suffix $\#\gamma'$ 
		by two positions to the right 
		and writes $a\#$ before it.  }        
\end{enumerate}
All described subroutines for Cases 1 can be done in linear time on one tape. 	  

\noindent {\bf Case 2}: 
Suppose that $s_0  \neq 0$. 
That is, $u$ is either of the form $u = a^{r_0} b^{s_0} 
\# \gamma$  for $\gamma \in \Sigma^*$ 
or $u = a^{r_0} b^{s_0}$, where $r_0 \geqslant 0$ 
and $s_0 > 0$. This can be verified in linear time 
by reading $u$. 
Since $w$ ends in 
$x_0^{-1}$, $w x_1 ^{-1}$ is not 
the infinite normal form for 
$g x_1 ^{-1}$. 
Applying the relations 
$x_0^{-1}x_j ^{-1} = 
x_{j+1}^{-1}x_0^{-1}$,  $j>0$, $f_0$ times 
we obtain that if 
$1+f_0 \leqslant j_1$,  
$g x_1^{-1}$ is equal to 
$  x_{i_0}^{e_0} x_{i_1}^{e_1} 
\dots x_{i_m}^{e_m} 
x_{j_n}^{-f_n} \dots 
x_{j_1} ^{-f_1} 
x_{1 + f_0}^{-1}
x_{0} ^{-f_0}$. 
If $1+f_0 > j_1$, then applying 
the relations 
$x_{j_1}^{-1} x_j ^{-1} = 
x_{j+1}^{-1}x_{j_1}^{-1}$, $j>j_1$, 
$f_1$ times we obtain that if 
$1+ f_0 + f_1 \leqslant j_2$, 
$gx_1 ^{-1}$ is equal to 
$x_{i_0}^{e_0} x_{i_1}^{e_1} 
\dots x_{i_m}^{e_m} 
x_{j_n}^{-f_n} \dots 
x_{j_2} ^{-f_2} 
x_{1 + f_0 + f_1}^{-1}
x_{j_1} ^{-f_1} 
x_{0} ^{-f_0}$.  
This process is continued until 
the first time we obtain $x_R^{-1}$ 
where either: 
\begin{itemize} 
	\item{$R = 1 + f_0 + f_1 + \dots + f_n > j_n$, or} 
	\item{$R = 1 + f_0 + f_1 + \dots + f_{t} 
		\leqslant j_{t+1}$ for some 
		$0 \leqslant t  
		\leqslant n-1$.} 	
\end{itemize}	 
In Algorithm \ref{alg_comp_R} below we describe a
two--tape Turing machine $\mathrm{TM}_R$ that for a given normal form $u$ written on the first tape writes a string $b^R$ on the second tape.   
\begin{algo}[A subroutine for computing $R$]
	\label{alg_comp_R} 
	
	Initially a content of the first tape is 
	$\boxplus u \boxdot^{\infty}$  with a head over 
	$\boxplus$. 
	A content of the second tape is 
	$\boxplus \boxdot^{\infty}$ with a head over $\boxplus$.    
	Let\, $\mathrm{STOP}_1$ be a boolean variable
	which is  true 
	if a head on the first tape 
	reads $\boxdot$ and false otherwise,
	let\, $\mathrm{STOP}_2$ 
	be a boolean variable which is true
	if a head on the second tape reads 
	$\boxplus$ and false otherwise.
	Finally let  $\mathrm{CASE}$ be a boolean 
	variable which is true if $R > j_n$ and false 
	if $R = 1 + f_0 + \dots + f_t 
	\leqslant j_{t+1}$ for some 
	$0 \leqslant t \leqslant n-1$.
	
	\begin{enumerate}
		\item{On the first tape 
			$\mathrm{TM}_R$ moves a head by one position to the right. On the second tape
			$\mathrm{TM}_R$ moves a head by one position  
			to the right and writes the symbol $b$.}		
		\item{While $\mathrm{not}$
			$(\mathrm{STOP_1}$ $\mathrm{or}$ 
			$\mathrm{STOP_2})$: 
			\begin{enumerate}[(a)]
				\item{If a head on the first tape reads the $a$ symbol,
					on the first tape $\mathrm{TM}_R$ moves the 
					head by one position to the right.}
				\item{If a head on the first tape reads the $b$ symbol,
					on the second tape $\mathrm{TM}_R$  
					moves the head  by one position
					to the right and writes the $b$ symbol while on the 
					first tape it moves a head by one position 
					to the right.} 
				\item{If a head on the first tape reads the $\#$    
					symbol, on the second tape $\mathrm{TM}_R$
					writes the blank symbol $\boxdot$ and moves a head
					by one position to the left while  on the first 
					tape it moves a head by one position to the right.} 	
			\end{enumerate}		 	 
		}  	
		\item{To find a correct value of the boolean 
			variable $\mathrm{CASE}$ the subroutine 
			proceeds as follows.	   
			\begin{enumerate}[(a)]
				\item{If\, $\mathrm{STOP_1}$, we set 
					$\mathrm{CASE = true}$.}
				
				\item{If\, $\mathrm{not}$ $\mathrm{STOP_1}$, 
					on the first tape $\mathrm{TM}_R$ 
					checks if a head reads the $b$ symbol.
					If not, on the first tape it moves a head by one position to the right and checks again if a head 
					reads the $b$ symbol. This process is 
					continued  until 
					either a head  reads 
					$\boxdot$ or the $b$ symbol. If it reads 
					$\boxdot$, we set $\mathrm{CASE =true}$. 
					If it reads the $b$ symbol, we set 
					$\mathrm{CASE = false}$.}                    
		\end{enumerate}}	
		
		\item{Then $\mathrm{TM}_R$ erases all 
			$b$ symbols on the 
			second tape. If a head on the second tape 
			reads the $b$ symbol it 
			writes $\boxdot$ and moves a head by 
			one position to the left. This process is 
			continued until a head on the second tape reads $\boxplus$.} 
		
		\item{ Depending on the value of 
			$\mathrm{CASE}$ the subroutine proceeds as follows. 
			\begin{enumerate}[(a)]	   
				\item{Suppose $\mathrm{CASE}$. 
					First a head on the 
					second tape moves by one position to the right 
					and writes the $b$ symbol. 
					Then, if a head on the first tape reads the 
					$b$ symbol, it moves by one position to the left 
					while a head on the second tape moves by one position 
					to the right and writes the $b$ symbol.
					If a head on the first tape reads $\#$ or 
					the $a$ symbol, it moves by one position 
					to the left. This process is continued until 
					a head on the first tape reads $\boxplus$.					
				}
				\item{Suppose $\mathrm{not}$ $\mathrm{CASE}$. 
					If a head on the first tape reads the 
					$b$ symbol, it moves by one position to the left 
					while a head on the second tape moves by one position 
					to the right and writes the $b$ symbol.
					If a head on the first tape reads $\#$ or 
					the $a$ symbol, it moves by one position 
					to the left. This process is continued until 
					a head on the first tape reads $\boxplus$. }	   	  
			\end{enumerate} 	    
		}	   		
	\end{enumerate}	
\end{algo}
From Algorithm \ref{alg_comp_R} it can be seen 
that a subroutine for computing $R$ can be done 
in linear time on a two--tape Turing machine. 
Depending on the value of $\mathrm{CASE}$ we consider 
the following two cases. 

\noindent {\bf Case 2.1}: 
Suppose $\mathrm{CASE}$. That is, $R>j_n$. 
Let $M = \max \{i_m, j_n\}$.  
There are three subcases to consider: 
$R> M$, $R= M$ and $R<M$. 
Each of these three subcases can be
checked as follows. 
First note that $M$ is just the number of $\#$ 
symbols in the normal form $u$.  
So we run a subroutine which reads $u$ 
on the first tape and each time a head 
reads the $\#$ symbol, on a separate 
tape it moves a head by one position 
to the right and writes the $\#$ symbol. 
In the end of this subroutine 
the content of this separate tape is 
$\boxplus \#^M \boxdot^{\infty}$.   
Now to check whether $R>M$, $R=M$ or 
$R<M$ we can synchronously read  
the tapes $\boxplus b^R \boxdot^{\infty}$ 
and $\boxplus \#^M \boxdot^{\infty}$
with the heads initially over the 
$\boxplus$ symbols. 

(a) Suppose $R>M$. Then the infinite normal 
form of $gx_1^{-1}$ is 
\[  
	x_{i_0} ^{e_0} x_{i_1}^{m_1} \dots 
	x_{i_m} ^{e_m} x_{R}^{-1} x_{j_n}^{-f_n}  
	\dots x_{j_2}^{-f_2} x_{j_1}^{-f_1} 
	x_0 ^{-f_0},  
\]  
where $f_0 \neq 0$.  
We write $u$ in the form
$u = a^{r_0} b^{s_0} \gamma$,  
where $\gamma$ is either empty or starts 
with $\#$, ends with $a$ or $b$ and 
contains exactly $M$ $\#$ symbols. 
Then  $v = a^{r_0} b^{s_0} 
\gamma \#^{R-M}b$.
A subroutine for computing 
$v$ from $u$ appends the string 
$\#^{R-M}b$ to $u$ as follows. 
First a head on the first tape where $u$ 
is written moves to the last 
non--blank symbol.   
Then we synchronously read the tapes $\boxplus b^R \boxdot^\infty$ and $\boxplus\#^M \boxdot^\infty$ from the beginning 
until both heads are over the $\boxdot$ 
symbols. If a head on the tape $\boxplus b^R \boxdot^\infty$ reads the $b$ symbol 
but a head on the tape $\boxplus\#^M \boxdot^\infty$ reads $\boxdot$, on 
the first tape  a head moves by one position to the right and writes the $\#$ symbol. 
As a result the content of a first tape 
will be 
$\boxplus a^{r_0}b^{r_0} \gamma \#^{R-M}\boxdot^{\infty}$ with a head over 
the last $\#$ symbol.  
After that a head on the first tape 
moves by one position to right and 
writes the $b$ symbol.

A subroutine for computing $u$ from $v$ 
moves a head to the last symbol of $u$, which is $b$. Then it writes the $\boxdot$ and moves a head by one position to the left. 
If a head reads $\#$  it writes 
$\boxdot$ and moves a head by one position 
to the left. This process is continued 
until a head reads a symbol which is not
$\#$. As a result the content of a first 
tape will be $\boxplus a^{r_0}b^{s_0}\gamma\boxdot^{\infty}$.

(b) Suppose $R = M$. This can 
only occur if $R= i_m$. Then the infinite 
normal form of $gx_1^{-1}$ is either: 
\[ 
	x_{i_0}^{e_0} x_{i_1}^{e_1} 
	\dots x_{i_m}^{e_m-1} 
	x_{j_n}^{-f_n} \dots x_{j_2}^{-f_2}
	x_{j_1}^{-f_1} x_{j_0}^{-f_0} \,\, 
	\mathrm{if} \,\, e_m>1, \mathrm{or}
\]	
\[    
	x_{i_0}^{e_0} x_{i_1}^{e_1} 
	\dots x_{i_{m-1}}^{e_{m-1}} 
	x_{j_n}^{-f_n} \dots x_{j_2}^{-f_2}
	x_{j_1}^{-f_1} x_{j_0}^{-f_0} \,\, 
	\mathrm{if} \,\,  e_m=1.
\]	
The latter expression 
is an infinite normal form 
as $i_m = R \geqslant j_n +2$ and $w$ is an infinite 
normal form.  
Indeed, for Case 2.1  we have that 
$1 +  f_0 + \dots + f_{n-1} > j_n$. Therefore, 
\[
	\label{R_lower_bound}  
	R = 1 + f_0 + \dots + f_{n-1} + f_{n} 
	\geqslant j_n + 2. 
\]
If we write $u$ in the form 
$u = \gamma \#^s a^{e_m}$, 
then 
$v = \gamma \#^s a^{e_m-1}$ when
$e_m > 1$ and $v = \gamma$ if $e_m=1$.  

A subroutine for computing $v$ from $u$ 
reads $u$ to check if $e_m>1$ or $e_m=1$. 
If $e_m>1$, it erases the last $a$ symbol 
by writing $\boxdot$.  If $e_m=1$, 
it erases the last $a$ symbol 
by writing $\boxdot$ and moves a head 
by one position to the left. 
If a head reads $\#$ it writes  
$\boxdot$  and moves by one position 
to the left. This is continued until  
a head reads a symbol which is not $\#$. 
As a result the content of the first 
tape will be $\boxplus \gamma \boxdot^{\infty}$. 

A subroutine for computing $u$ from $v$ 
is a follows. First we run 
Algorithm \ref{alg_comp_R} for the input $v$. 
As result we get $\boxplus b^R \boxdot^{\infty}$
written on a second tape.
Now let $M'$ be the 
number of $\#$ symbols in $v$. Like in 
the subcase (a) of Case 2.1 we run a
subroutine that computes $M'$ and appends
$\#^{R-M'}$ to $v$ if $R \geqslant M'$; 
if $R<M'$ the consideration of this 
subcase (b) is skipped. 
Finally in the last step it appends $a$. 
As a result the content of the first tape
will be $\boxplus \gamma \#^{R-M'}a 
\boxdot^{\infty}$ if $R>M'$ and 
$\boxplus \gamma a \boxdot^{\infty}$ if $R=M'$.

(c) Suppose $R<M$. Then 
$i_m = M$. There are three subsubcases to consider.  

  1) The generator $x_R$ does not
appear in $w$. That is, 
$u$ is of the form
$u = \gamma \# a^{r_{R-1}} \# \# \eta$, 
where $\eta \in \{a, \#\}^*$;  
note that $s_{R-1}=0$ by the inequality 
\eqref{R_lower_bound}. 
This subsubcase is verified by finding the  
$R$th $\#$ symbol in $u$ and checking whether 
or not the next symbol after it is $a$.  
If it is not $a$, 
then $x_R$ does not appear in $w$.
If we write  
$u$ in the form
$u = \gamma \# a^{r_{R-1}} \# \# \eta$,
then 
$v = \gamma \# a^{r_{R-1}} 
\# b \# \eta$.   
A subroutine for computing $v$ from $u$ 
inserts the $b$ symbol before the $(R+1)$th 
$\#$ symbol -- it shifts the suffix $\# \eta$,
which begins with the $(R+1)$th $\#$ symbol,  
by one position to the right and
writes the $b$ symbol before it.  
A subroutine for computing $u$ from $v$ shifts 
the suffix $\#\eta$ by one position to the left
erasing the $b$ symbol. This can be done without 
knowing $b^R$ written on another tape: 
we read $v$ from the right to the left until 
a head reads the $b$ symbol, then  we shift 
the suffix $\# \eta$ following it by one 
position to the left. 
This is a correct algorithm   
as the suffix $\eta$ does not have any $b$ symbols.

2) The generator $x_R$ appears in $w$ together with $x_{R+1}$. 
That is, $u$ is either of the form 
$u=\gamma \# a^{r_R} \# a^{r_{R+1}}   \eta$,
with $r_R, r_{R+1} >0$ , where 
$\eta \in \{a, \#\}^*$ is either empty or begins with $\#$. 
This subsubcase is verified 
by finding the $R$th and $(R+1)$th $\#$ symbols in $u$ 
and checking whether or not both symbols after them are $a$. 
If both of them are $a$, then $x_R$ and $x_{R+1}$ 
appear in $w$. 
If we write $u$ in the
form $u=\gamma \# a^{r_R} \# a^{r_{R+1}}
\eta$, 
then $v = \gamma \# a^{r_R} b\# a^{r_{R+1}}
\eta$.
A subroutine for computing $v$ from $u$ inserts the 
$b$ symbol before the $(R+1)$th $\#$ symbol like in the 
previous subsubcase. A subroutine for computing 
$u$ from $v$ shifts the suffix 
$\# a^{r_{R+1}} \eta$ 
by one position to
the left erasing the $b$ symbol. 
Like in the previous case 
this can be done without 
knowing $b^R$ written on another tape: 
we read $v$ from the right to the left until 
a head reads the $b$ symbol, then  we shift 
the suffix following it by one 
position to the left. 

3) The generator $x_R$ appears in $w$ but $x_{R+1}$ does not appear in $w$. 
That is, $u$ is of the form 
$u = \gamma \# a^{r_R} \# \#\eta$, where 
$\eta \in \{a, \# \}^{*}$. 
This subsubcase 
is verified by checking whether or not the first symbol after 
the $R$th $\#$ symbol is $a$ and the first symbol after 
the $(R+1)$th $\#$ symbol is not $a$. If the latter is 
true, then $x_R$ appears in $w$ but $x_{R+1}$ does not. 
Now if $u = \gamma \# a^{r_R} \# \#\eta$, 
where $\eta \in \{a, \#\}^*$, 
then $v= \gamma \# a^{r_R-1} \# \eta$; 
note that for $r_R=1$, $v$ is a valid normal form by the inequality \eqref{R_lower_bound}.   
A subroutine for computing $v$ from $u$ shifts the suffix
$\# \eta$ following the $(R+1)$th $\#$ symbol by two
positions to the left erasing the subword $a\#$ that precedes 
this suffix.
A subroutine for computing $u$ from $v$ 
first runs Algorithm \ref{alg_comp_R} 
for the input $v$
writing  $\boxplus b^R \boxdot^{\infty}$
on a second tape 
and then inserts the subword $a\#$ before 
the suffix $\#\eta$ which begins 
with the $(R+1)$th $\#$ symbol; if the 
$(R+1)$th $\#$ symbol is not found in $v$ then 
the consideration of this 
subsubcase  is skipped.

\noindent {\bf Case 2.2}: Suppose $\mathrm{not}$ 
$\mathrm{CASE}$.  That is,  
$R = 1+ f_0 + f_1 + \dots + f_t  \leqslant j_{t+1}$ for some 
$0 \leqslant t \leqslant n-1$. 
Then $gx_1^{-1}$ is equal to:  
\[	
	\label{case22gx1}   
	x_{i_0}^{e_0} x_{i_1}^{e_1} \dots x_{i_m}^{e_m} 
	x_{j_n}^{-f_n} \dots x_{j_{t+1}}^{-f_{t+1}} 
	x_R ^{-1} x_{j_t} ^{-f_t} \dots x_{j_0}^{-f_0}. 
\]	
Note that $R> j_t$ by the construction of $R$.  
There are two cases to consider depending whether or not \eqref{case22gx1} is an infinite normal form.

\noindent {\bf Case 2.2.1}: Suppose that \eqref{case22gx1} is not an infinite 
normal form. This only happens when $R<j_{t+1}$ and 
for $w$ 
the generator $x_R$ is present to a positive power while $x_{R+1}$ is not present to any non--zero power. 
This situation occurs only if   
there is an index $p \leqslant m$ for which 
$i_p=R$, $i_{p+1} \neq R+1$ and $j_{t+1} \neq R+1$. That is, 
$u$ is of the form 
$u = \gamma \# a^{r_R} \# \# \eta$. 
This subcase is verified by finding 
the $(R+1)$th $\#$ symbol
and checking if the previous symbol 
is $a$ and the next symbol is $\#$.    
If we  write $u$ in the form 
$u = \gamma \# a^{r_R} \# \# \eta$, then  
$v = \gamma \# a^{r_R-1} \# \eta$.   
Note that for $r_R = 1$, 
$v$ is a valid normal form 
since $R > j_t +1$; 
this is because 
$1+ f_0 + \dots + {f_{t-1}}  > j_t$, so  
$R = 1 + f_0 + \dots + f_{t-1} + f_t > j_t +1$. 

A subroutine for computing $v$ from $u$ 
shifts the suffix $\# \eta$ following 
the $(R+1)$th $\#$ symbol by two 
positions to the left erasing the subword 
$a \#$ that precedes this suffix. 
A subroutine for computing $u$ from $v$ is the
same as for the subsubcase (c).3 of Case 2.1: 
it runs Algorithm \ref{alg_comp_R} for the input 
$v$ and then inserts the subword 
$a \#$ before the suffix $\# \eta$ 
which begins with the $(R+1)$th $\#$ 
symbol. Note that for Case 2.2.1 
$R$ must be the same for $u$ and $v$ 
as $R<j_{t+1}$.

\noindent {\bf Case 2.2.2}: Suppose that \eqref{case22gx1} is an infinite 
normal form. This happens only in the following subcases. 
\begin{enumerate}[(a)]	
	\item{$x_R^{-1}$ is already in $w$, that is, 
		$R = j_{t+1}$. 
		That is, $u$ is of the form 
		$u = \gamma \# a^{r_R} b^{s_R}  \eta$ 
		with $s_R>0$, where $\eta$ is either empty 
		or begins with $\#$. 
		This subcase is verified by finding 
		the $R$th $\#$ symbol in $u$ and 
		checking if the suffix following it 
		is of the form  
		$a^kb \mu$ for $k \geqslant 0$.  
		If we write $u$ in the form 
		$u = \gamma \# a^{r_R} b^{s_R} \eta$, 
		then $v = \gamma \# a^{r_R} b^{s_R+1} \eta$. 
		A subroutine for computing $v$ from $u$ 
		first reads the input $u$ until it finds the 
		$R$th $\#$ symbol. Then it reads the suffix 
		$a^{r_R} b^{s_R} \eta$ until it reads 
		the $b$ symbol first time. 
		After that it shifts the suffix  
		$b^{s_R}  \eta$ by one position to 
		the right and writes the $b$ symbol before it.
		A subroutine for computing $u$ from $v$ 
		first reads the input $v$ until it finds 
		the $R$th $\#$ symbol. Then it reads 
		the suffix $a^{r_R} b^{s_R+1}  \eta$ 
		until 
		it reads the $b$ symbol first time. After that 
		it erases this $b$ symbol and shifts  
		the suffix following it by one 
		position to the left.}	
	\item{$x_R^{-1}$ is not in $w$, but $x_R$ and either 
		$x_{R+1}$ or $x_{R+1}^{-1}$ are present in $w$. 
		That is, $u$ is  of the form 
		$u = \gamma \# a^{r_R} \# a^{r_{R+1}} 
		b^{s_{R+1}} \eta$ 
		with $r_R > 0$ and 
		$r_{R+1} + s_{R+1}>0$, where $\eta$ is either 
		empty of begins with $\#$. 
		This subcase is verified by finding the 
		the $R$th $\#$ symbol in $u$ and checking 
		if the suffix following it is of the form 
		$a^k \# a \mu $ or 
		$a^k \# b \mu$ for $k \geqslant 1$.        	   
		If we write $u$ in the form $u = \gamma \# a^{r_R} \# a^{r_{R+1}}   b^{s_{R+1}}  \eta$, then	                
		$v = \gamma \# a^{r_R} b\# a^{r_{R+1}} 
		b^{s_{R+1}} \eta$.
		A subroutine for computing $v$ from $u$ 
		inserts the $b$ symbol before the 
		suffix 
		$\# a^{r_{R+1}}   b^{s_{R+1}}  \eta$
		which begins with the 
		$(R+1)$th $\#$ symbol.   
		A subroutine for computing $u$ from $v$
		shifts the suffix             
		$\# a^{r_{R+1}}  b^{s_{R+1}}  \eta$ 
		by one position to the left which erases the 
		$b$ symbol preceding this suffix.}
	\item{Both $x_R$ and $x_{R}^{-1}$ are not present in $w$. 
		That is, $u$ is of the form 
		$u = \gamma \# a_{r_{R-1}}  \# \# \eta$; 
		note that $b_{R-1}=0$ because $R> j_t+1$.
		This subcase is verified by finding the 
		$R$th $\#$ symbol in $u$ and checking that
		the symbol next to it is $\#$. 
		If we write $u$ in the form 
		$u = \gamma \# a_{r_{R-1}}  \# \# \eta$, then $v = \gamma \# a_{r_{R-1}}  \# b \# \eta$. 
		A subroutine for computing $v$ from $u$
		inserts the $b$ symbol  after the $R$th 
		$\#$ symbol. A subroutine for 
		computing $u$ from $v$ shifts the suffix 
		$\# \eta$ which begins with 
		the $(R+1)$th $\#$ symbol by one position 
		to the left which erases the $b$ symbol preceding 
		this suffix.}	               
\end{enumerate}	
All described subroutines for Case 2 can be done in linear time on two tapes. Indeed, Algorithm \ref{alg_comp_R} 
requires only two tapes with the output 
$\boxplus b^R \boxdot^{\infty}$ appearing on the second 
tape. Furthermore, in all subroutines where we needed 
an extra tape we could use the convolution of the second tape
and this extra tape. 
When we use a separate tape to compute 
$M$, writing $\boxplus \#^M \boxdot^{\infty}$ on it, 
we can simply do it on the second tape using the symbols  
$b \choose \#$,
$b \choose \boxdot$ and
$\boxdot \choose \#$.  
For the same argument in the construction of 
$\mathrm{TM}_{x_1}$ introducing the additional tape 
for storing a copy of $v$ can be avoided.
Thus we proved the following theorem. 

\begin{thm} 
	\label{F_is_2_tape_thm}	
	Thompson's group $F$ is Cayley 2--tape linear--time 
	computable.
\end{thm}	

\section{Discussion and Open Questions}  
\label{conclusion_section}

Theorems \ref{Z2wrF2_2_tape_thm} and \ref{F_is_2_tape_thm}  
show that the wreath product 
$\mathbb{Z}_2 \wr \mathbb{F}_2$
and Thompson's group $F$ admit 
quasigeodesic $2$--tape linear--time 
computable normal forms.  
The following questions are apparent from
these results.  

\begin{enumerate} 
	\item[1.]{Is $F$ Cayley position--faithful (one--tape) 
		linear--time computable?}	
	\item[2.]{Is $\mathbb{Z}_2 \wr \mathbb{F}_2$ 
		Cayley position--faithful (one--tape) 
		linear--time computable?}	
\end{enumerate}	
It is an open problem whether or not $F$ is automatic. The first question is a weak formulation 
of this open problem. The group 
$\mathbb{Z}_2 \wr \mathbb{F}_2$ is not automatic.
However, it is not known whether or not 
$\mathbb{Z}_2 \wr \mathbb{F}_2$ is Cayley automatic. 
The second question is a weak formulation of
the latter problem.  
Theorem \ref{Z2wrZ2_2_tape_thm} shows that
the wreath product $\mathbb{Z}_2 \wr \mathbb{Z}^2$ 
admits a $2$--tape linear--time computable normal 
form.   
However, this normal form is not 
quasigeodesic\footnote{Though this normal 
	form is not quasigeodesic, one can show that 
	there is an algorithm computing 	
	it in quadratic time.}.
It is an open problem whether or not 
$\mathbb{Z}_2 \wr \mathbb{Z}^2$ is Cayley 
automatic. As a weak formulation 
of this open problem we leave the following
question for future consideration. 
\begin{enumerate} 
	\item[3.]{Does 
		$\mathbb{Z}_2 \wr \mathbb{Z}^2$ admit 
		a quasigeodesic normal form 
		for which the right multiplication by a 
		group element is computed in polynomial time?}
\end{enumerate}	
By Theorem \ref{onetape_implies_quasigeodesic_thm}, 
if for a normal form the right multiplication 
is computed on a one--tape Turing machine in 
linear time, then it is always quasigeodesic.  
So when studying extensions of Cayley
automatic groups it sounds natural 
to restrict oneself to quasigeodesic normal forms.       
We leave the following extensions of Cayley automatic
groups for future 
consideration\footnote{Adding the class 
	of Cayley linear--time computable groups 
	refines the Venn diagram of extensions 
	of interest shown in \cite[Fig.~1]{BEK21}.}: 
\begin{itemize} 
	\item{Cayley position--faithful (one--tape) linear--time 
		computable groups;}
	\item{Cayley linear--time computable groups 
		with quasigeodesic normal form;}    
	\item{Cayley polynomial--time computable groups 
		with quasigeodesic normal form.}        	
\end{itemize}	
This paper considers only the complexity of the 
right multiplication by a group element. We leave studying 
the complexity of the left multiplication for 
future work.

\section*{Acknowledgment}
  \noindent 
  The authors thank the anonymous reviewer for useful 
  comments. The authors wish to acknowledge fruitful discussions with Murray Elder.
\bibliographystyle{plain}

\bibliography{cayleylinear}
\end{document}